\documentclass{article}
\usepackage[utf8]{inputenc}
\usepackage{latexsym}
\usepackage{amsmath}
\usepackage{amssymb}
\usepackage{amsthm}
\newtheorem{theorem}{Theorem}[section]

\newtheorem{lemma}[theorem]{Lemma}

\newtheorem{corollary}[theorem]{Corollary}
\newtheorem{question}[theorem]{Question}

\newtheorem{definition}[theorem]{Definition}

\newtheorem{rmk}[theorem]{Remark}

\newtheorem{example}[theorem]{Example}

\newtheorem{prop}[theorem]{Proposition}
\usepackage{todonotes}
\usepackage{mathtools}
\usepackage{hyperref}
\usepackage{xcolor}
\usepackage{float}
\usepackage{pgfplots}
\usepackage{caption}  
\pgfplotsset{compat = newest}

\makeatletter
\hypersetup{
	backref,colorlinks=true,linkcolor=blue,citecolor=blue,urlcolor=blue,citebordercolor={0 0 1},urlbordercolor={0 0 1},linkbordercolor={0 0 1}}
 
\newcommand{\ZZ}{\mathbb{Z}}
\newcommand{\QQ}{\mathbb{Q}}
\newcommand{\CC}{\mathbb{C}}
\newcommand{\RR}{\mathbb{R}}
\newcommand{\PP}{\mathbb{P}}

\title{Algebraic capacities as tropical polynomials over the reduced $c_1$-positive symplectic cone}
\author{Tian-Jun Li, Shengzhen Ning }
\AtEndDocument{\bigskip{\footnotesize
		\textsc{School of Mathematics, University of Minnesota, Minneapolis, MN, US} \par
		\textit{E-mail address}: \texttt{tjli@math.umn.edu} \par
		\addvspace{\medskipamount}
		\textsc{School of Mathematics, University of Minnesota, Minneapolis, MN, US} \par
		\textit{E-mail address}: \texttt{ning0040@umn.edu} \par
		
}}
\date{\today}

\begin{document}

\maketitle

\begin{abstract}
  In a series of work \cite{Algebraic}, \cite{BenEhrhart} and \cite{rationalobs}, algebraic capacities were introduced in an algebraic manner for polarized algebraic surfaces and applied to the symplectic embedding problems. In this paper, we give a reformulation of algebraic capacities in terms of only a tamed pair of symplectic form and almost complex structure. We show that they actually only depend on the cohomology class of the symplectic form for a rational manifold. Since it is not known that any symplectic form on a rational manifold is K\"{a}hler, this novel formulation potentially is more general on a rational manifold. Additionally, for manifolds with $b_2^+=1$, we derive asymptotic results that are parallel to the context of ECH (Embedded Contact Homology) and algebraic settings. When assuming $c_1\cdot [\omega]>0$ on rational manifolds, we further introduce a sequence of tropical polynomials which will succinctly describe those capacities viewed as functions over the domain parametrizing such symplectic forms. As an application, we give a purely symplectic proof of the correspondence between algebraic capacities and ECH capacities for smooth toric surfaces.
\end{abstract}

\tableofcontents

\section{Introduction}

 For a normal projective algebraic surface $Y$ with a big $\QQ$-Cartier $\RR$-divisor $A$, \cite{Algebraic} introduces the $k$-th algebraic capacity in its full generality as
$$c_k^{\rm{alg}}(Y,A):=\inf_{D\in \text{Nef}^{\text{qc}}(Y)}\{D\cdot A\,|\,\chi(D)\geq k+\chi(\mathcal{O}_Y)\},$$
  where $\text{Nef}^{\text{qc}}(Y)$ denotes the set of nef $\QQ$-Cartier $\ZZ$-divisors on $Y$. When $Y$ is smooth and $A$ is ample, the above definition can be read as
$$c_k^{\rm{alg}}(Y,A):=\inf_{D\in \text{Nef}(Y)_\ZZ}\{D\cdot A\,|\,D^2-K_Y\cdot D\geq 2k\},$$
where $K_Y$ is the canonical class of $Y$. In this paper, we delve into the so-called {\bf tamed triple} $(X,\omega,J)$, where $X$ is a closed $4$-manifold possessing a symplectic form $\omega$ and a tamed almost complex structure $J$. We introduce the following quantitative invariant $f_k(X,\omega,J)$ for each positive integer $k$, which we call the {\bf $k$-th tamed capacity}, to be  $$\inf \{\omega(A)\,|\,A\in H_2(X;\ZZ)\cap \text{PD}(H^{+}_J(X)),A\text{ is } J\text{-nef}, ind(A)\geq 2k\}\footnote{We adopt the convention that the infimum over the empty set is $\infty$.}.$$ 
Here $H^{+}_J(X)=H^{1,1}_J(X)\cap H^2(X;\RR)$, introduced in \cite{comparetamecompatible}, is the space of $J$-invariant cohomology classes housing almost K\"{a}hler classes, where $H^{1,1}_J(X)\subseteq H^2(X;\CC)$ is the almost complex analogue of the Dolbeault cohomology group. The term  ``$J$-nefness" is the almost complex analogue of nefness in algebraic geometry, while $ind(A)$ denotes the real Gromov-Taubes dimension $A^2-K\cdot A$ where $K=K_{\omega}=K_J$ denotes the canonical class associated to $\omega$ or $J$. Further elucidation of these definitions will be provided in subsequent sections. When $(\omega,J)$ is a K\"{a}hler pair coming from a projective embedding of $X$, this definition naturally aligns with the $k$-th algebraic capacity $c_k^{\rm{alg}}$ for a smooth polarized surface in \cite{Algebraic}. Note that $c_k^{\rm{alg}}$ can be defined for possibly non-smooth algebraic surfaces, and $f_k$ can be defined for possibly non-integrable $J$ which is only assumed to be tamed by $\omega$. Since there are also plenty of examples of non-K\"{a}hler (thus non-algebraic) symplectic manifolds (\cite{KodairaThurston},\cite{Draghici},\cite{nonkahlerruled},\cite{nonkahlerT4}), our $f_k$'s are more general in this sense.

Utilizing the machinery of embedded contact homology (ECH), \cite{rationalobs} gives compelling applications concerning obstructions to symplectic embeddings into polarized algebraic surfaces, particularly toric surfaces, highlighting the significance of the quantities $c_k^{\rm{alg}}$'s. The application of algebraic geometry to study symplectic embedding problems traces its roots back to the seminal work of \cite{MPpacking}. However, the utilization of algebraic geometry diverges substantially between \cite{MPpacking} and \cite{rationalobs}. In \cite{MPpacking}, Nakai-Moishezon criterion and the understanding of curves on certain surfaces are applied to produce K\"{a}hler forms, leading to the existence of embeddings through the equivalence established between ball packings and the existence of symplectic forms on the blowup manifold. Conversely, in \cite{rationalobs}, those $c_k^{\rm{alg}}$'s are defined through an algebraic process and applied to give obstructions to embeddings. Our motivation of introducing the tamed capacities $f_k$'s is to explore the extent of what can be achieved by circumventing certain technical and deep results in algebraic geometry. 

The definition of our tamed capacities heavily depends on the almost complex structure $J$ since the space $H^{+}_J(X)$ is an almost complex invariant which already relies on $J$. But a significant result highlighted in \cite[corollary 3.4]{DTZ10} establishes that when the manifold has $b_2^+=1$, $H^{+}_J(X)$ encompasses the entire $H^2(X;\RR)$. Let's assume the $b_2^+=1$ condition for the moment. Since any torsion class in $H_2(X;\ZZ)$ can not have index $\geq 2k\geq 2$, we can just use $H_2(X;\ZZ)$ instead of $H_2(X;\ZZ)\cap \text{PD}(H^{+}_J(X))$ in the definition of $f_k$. Note that we are still doing optimization on a $J$-dependent set $$U_1:=\{A\text{ is } J\text{-nef}, ind(A)\geq 2k\}\subseteq H_2(X;\ZZ).$$ However, when considering rational manifolds\footnote{We use the terminology rational manifold rather than rational surface throughout the paper. This means the manifold does not come with a complex structure. Hopefully this can avoid unnecessary confusion from readers of more algebraic inclinations. } $\CC\PP^2\#n\overline{\CC\PP}^2$, with the assistance of Seiberg-Witten and Gromov-Taubes theories we can introduce other six subsets $U_2,\cdots,U_7$ (see section \ref{section:hierarchy} for definitions) of $\{ind\geq 2k\}\subseteq H_2(X;\ZZ)$ which all bear some resemblance to $U_1$. These seven sets, although defined differently, share an intriguing and harmonious relationship. 

\begin{prop}[=Proposition \ref{prop:hierarchy}]
    $$U_1\subsetneq U_2\subsetneq U_3\subsetneq U_4= U_5=U_6=U_7.$$ 
\end{prop}

This result is an improvement of \cite[proposition 2.5]{rationalobs}. Within these seven sets, $U_5$ holds significant importance from the perspective of computations since it's simply given by $$U_5:=\{A\cdot H>0, ind(A)\geq 2k\}\subseteq H_2(X;\ZZ).$$
One can easily see that $U_5$ is purely homological and thus $J$-independent. Our tamed capacities then become computable due to the following result.

\begin{theorem}[=Proposition \ref{prop:equivalentdef}+Corollary \ref{cor:indepofj}]\label{thm:introequivdef}
    If $X$ is a rational manifold $\CC\PP^2\#n\overline{\CC\PP}^2$, $f_k$ can be defined in seven equivalent ways by doing optimizations on $U_1,\cdots,U_7$. The value of $f_k$ only depends on the cohomology class $[\omega]$.
\end{theorem}
 Note that the work \cite{mcduffsiegel,HutchingsElementary} use a max-min procedure to give an elementary alternative definitions for symplectic capacities so as to avoid heavy Floer-theoretic machinery in symplectic geometry. Therefore if we let $\mathcal{J}(X,\omega)$ denote the space of all tamed almost complex structures and study the invariant $\sup_{J\in \mathcal{J}(X,\omega)} f_k(X,\omega,J)$ in alignment with the philosophy demonstrated in \cite{mcduffsiegel,HutchingsElementary}, the above theorem tells that taking the supremum over all tamed almost complex structures is redundant for rational manifolds.

 In general, the preceding $J$-independence result is too strong to hold true for all manifolds (see example \ref{rmk:exampleT4}). Consequently, our focus will be narrowed to $b_2^+=1$ manifolds and firstly investigate their asymptotic behaviors. The asymptotic behaviors of ECH capacities were first conjectured in \cite{Hutquantitative} and later established in \cite{ECHasymp}. Its algebraic analogue was then pursued in \cite{Algebraic}. To be more precise, in the algebraic settings it's shown that $c_k^{\rm{alg}}(Y,A)$ behave like $\sqrt{2A^2k}$ as $k\rightarrow\infty$ when $Y$ is either smooth or toric and $A$ is big and nef. We demonstrate that the argument in \cite{Algebraic} can be adapted to our settings, allowing us to establish the asymptotic behavior under the assumption $b_2^+=1$. 
 \begin{prop}[=Proposition \ref{prop:asymp}]\label{prop:asympintro}
     Let $(X,\omega,J)$ be a tamed triple with $b_2^+(X)=1$, then
    \[\lim_{k\rightarrow\infty}\frac{f_k(X,\omega,J)^2}{k}=2[\omega]^2.\]
 \end{prop}

 Again we have to emphasize that the preceding result, although proven in a manner resembling the algebraic settings, does not directly follow as an immediate consequence due to the examples of non-K\"{a}hler symplectic structures. The assumption $b_2^+(X)=1$ here is vital, which prompts us to speculate an asymptotic version of $J$-independence for $f_k(X,\omega,J)$, aiming to extend theorem \ref{thm:introequivdef} from rational manifolds to all $b_2^+(X)=1$ cases. See question \ref{question:asym} for the precise formulation.
 
 Now we return to the scenario where $X$ is a rational manifold. Then $f_k$ can naturally be viewed as a function over the region in $H^2(X;\RR)$ that accommodates symplectic classes. Since it's a long-standing unsolved question whether any symplectic form on a rational manifold is K\"{a}hler, this region potentially could be larger than any ample cone of $X$ when $X$ is equipped with some complex structure and regarded as an algebraic surface. The next main result of this paper is that after further assuming the {\bf $c_1$-positive} condition $c_1\cdot[\omega]>0$, the pattern of $f_k$ can be succinctly described by {\bf tropical polynomials}. It's worth noting that this condition includes all symplectic toric surfaces, given that the preimage of moment map polygon boundary forms a symplectic divisor representing the class $c_1$. It also includes all $\CC\PP^2\#n\overline{\CC\PP}^2$ with $n\leq 9$ by light cone lemma (\cite[lemma 3.7]{McDufflecture}). 

\begin{rmk}
     Unlike the works such as \cite{casalvianna} which essentially employ tropical techniques to construct curves, our formulation involving tropical polynomials is purely for the sake of convenience. Readers who find this unfamiliar or discomforting can seamlessly substitute `tropical' with `finitely piecewise linear'.
 \end{rmk}
 
Recall that the tropical addition and multiplication are given by 
$$x\oplus y:=\min\{x,y\},\quad\quad\quad x\odot y:=x+y. $$
Let $B\subset\ZZ^n$ be a finite set and $c_{\vec{a}}$ a real number for each $\vec{a}\in B$. Under this redefined arithmetic operation, a tropical polynomial is
\[f(x_1,\cdots,x_n):=\sum_{\vec{a}\in B}c_{\vec{a}}\vec{x}^{\vec{a}}=\text{min}\{c_{\vec{a}}+\langle{\vec{x}},\vec{a}\rangle\,|\,\vec{a}\in B\}.\]

The first interesting example is $\CC\PP^2\#\overline{\CC\PP}^2$. In this case the domain of $f_k$ has two parameters $\omega(H)$ and $\omega(E)$ where $H$ is the line class and $E$ is the exceptional class. It's convenient to only consider the scaled symplectic forms with $\omega(H)=1$ so that the domain can be parametrized by a single variable $x:=\omega(E)\in (0,1)$. This actually involves taking a {\bf normalized} cross section of the symplectic cone (see section \ref{section:symplecticcone} and the figure therein). Let $f_k(x)=f_{k}(X,[\omega])$ with $\omega(H)=1, \omega(E)=x$. Then with a bit of effort one can write down the first few terms (see figure \ref{fig:graph}):
\[f_1(x)=1\odot x^{-1},\] \[f_2(x)=(2\odot x^{-2}) \oplus 1,\] \[f_3(x)=(3\odot x^{-3})\oplus (2\odot x^{-1}),\] 
\[f_4(x)=(4\odot x^{-4})\oplus(2\odot x^{-1}),\] \[f_5(x)=(5\odot x^{-5})\oplus(3\odot x^{-2})\oplus 2,\]
\[f_6(x)=(6\odot x^{-6})\oplus(3\odot x^{-2}),\]
\[f_7(x)=(7\odot x^{-7})\oplus(4\odot x^{-3})\oplus(3\odot x^{-1}),\] \[f_8(x)=(8\odot x^{-8})\oplus(4\odot x^{-3})\oplus(3\odot x^{-1}), \]\[\cdots\cdots\]

However there is no guarantee that all $f_k$ are indeed tropical polynomials when $k$ is getting large. Namely, the functions could be piecewise linear but have infinitely many non-smooth points. To effectively utilize the tropical framework for describing these functions, it becomes imperative to demonstrate the finiteness of the tropical sum. Consequently, our primary theorem can be interpreted as follows:
\begin{theorem}[=Theorem \ref{thm:main}]\label{thm:intromain}
    For rational manifold $X$, there exists finitely many minimizers $A_1,\cdots,A_m$ such that $f_k(X,[\omega])=\min\{\omega(A_1),\cdots,\omega(A_m)\}$ for all $[\omega]\in\mathcal{P}^{c_1>0}$, where $\mathcal{P}^{c_1>0}$ denotes a polyhedral domain naturally parametrizing all symplectic forms satisfying $c_1\cdot[\omega]>0$ which will be introduced in section \ref{section:symplecticcone}.
\end{theorem}
 We emphasize that the tropical property does not hold in the absence of the $c_1$-positive condition, as illustrated in the example provided in remark \ref{rmk:nonc1nef}. Working in the algebraic category, \cite{Algebraic} obtains the locally finiteness when viewing $c_k^{\rm{alg}}$ as a function over the big cone. The aforementioned theorem can thus be understood as a global finiteness over a large region parametrizing $c_1$-positive symplectic forms. We refer the readers to section \ref{section:comparison} for a more detailed discussion.

 For the rational manifold $X$, if we view $\frac{f_k(X,[\omega])^2}{k}$ as a function over $\tilde{\mathcal{P}}^{c_1>0}$, the cross section of $\mathcal{P}^{c_1>0}$ by normalizing $\omega(H)=1$, then proposition \ref{prop:asympintro} confirm the pointwise convergence on the bounded yet non-compact domain $\tilde{\mathcal{P}}^{c_1>0}$. But now the tropical property allows us to extend the functions $\frac{f_k(X,[\omega])^2}{k}$ continuously onto the closure of $\tilde{\mathcal{P}}^{c_1>0}$. Therefore it could be viewed as a strengthened version of the asymptotic behavior from pointwise convergence to uniform convergence (see figure \ref{fig:graph}).
 
 
\begin{figure}[h]
\centering
\begin{tikzpicture}[scale=0.70]
\begin{axis}[
    xmin = 0, xmax = 1,
    ymin = 0, ymax = 3.1]
    \addplot[
        domain = 0:1,
    ] {1-x};
     \addplot[draw=red,
        domain = 0:1,
    ] {min(2-2*x,1)};
    \addplot[draw=orange,
        domain = 0:1,
    ] {min(3-3*x,2-x)};
      \addplot[draw=yellow,
        domain = 0:1,
    ] {min(4-4*x,2-x)};
    \addplot[draw=green,
        domain = 0:1,
    ] {min(5-5*x,3-2*x,2)};
    \addplot[draw=blue,
        domain = 0:1,
    ] {min(6-6*x,3-2*x)};
    \addplot[draw=cyan,
        domain = 0:1,
    ] {min(7-7*x,4-3*x,3-x)};
    \addplot[draw=violet,
        domain = 0:1,
    ] {min(8-8*x,4-3*x,3-x)};
\end{axis}
\end{tikzpicture}
\begin{tikzpicture}[scale=0.70]
\begin{axis}[
    xmin = 0, xmax = 1,
    ymin = 0, ymax = 3.1]
    \addplot[
        domain = 0:1,
    ] {(1-x)*(1-x)};
     \addplot[draw=red,
        domain = 0:1,
    ] {min((2-2*x)*(2-2*x),1)/2};
    \addplot[draw=orange,
        domain = 0:1,
    ] {min((3-3*x)*(3-3*x),(2-x)*(2-x))/3};
      \addplot[draw=yellow,
        domain = 0:1,
    ] {min((4-4*x)*(4-4*x),(2-x)*(2-x))/4};
    \addplot[draw=green,
        domain = 0:1,
    ] {min((5-5*x)*(5-5*x),(3-2*x)*(3-2*x),2*2)/5};
    \addplot[draw=blue,
        domain = 0:1,
    ] {min((6-6*x)*(6-6*x),(3-2*x)*(3-2*x))/6};
    \addplot[draw=cyan,
        domain = 0:1,
    ] {min((7-7*x)*(7-7*x),(4-3*x)*(4-3*x),(3-x)*(3-x))/7};
    \addplot[draw=violet,
        domain = 0:1,
    ] {min((8-8*x)*(8-8*x),(4-3*x)*(4-3*x),(3-x)*(3-x))/8};
    \addplot[draw=brown,
    line width=3.3pt,
    dashed,
        domain = 0:1,
    ] {2-2*x*x};
\end{axis}
\end{tikzpicture}
\captionof{figure}{The left figures are the first eight capacity functions $f_1,\cdots,f_8$. The right figures are $\frac{f_1^2}{1},\cdots,\frac{f_8^2}{8}$. The functions $\{\frac{f_k^2}{k}\}_{k\in{\ZZ_+}}$ will uniformly converge to the volume function $2[\omega]^2$ shown by the dashed curve.}\label{fig:graph}
\end{figure}
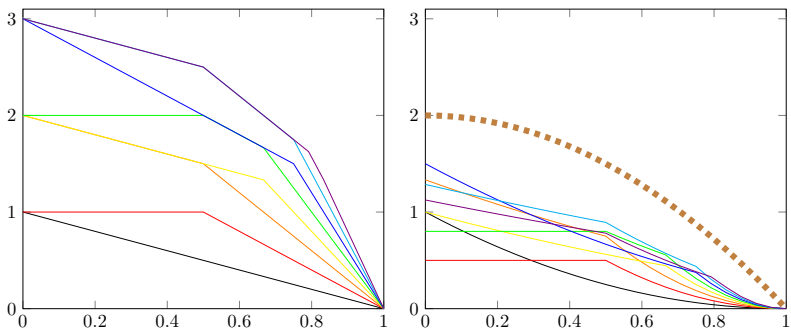

 Besides its own intriguing phenomenon, we find that this new tropical interpretation also has an application in establishing the relation with ECH capacities. It has been shown in \cite{BenEhrhart} that the algebraic capacities of a closed toric surface are closely related to the ECH capacities of its associated convex toric domain (for definition, see section \ref{section:ECH}). The original proof in \cite{BenEhrhart} entailed terminologies from algebraic toric geometry. As an application of our tropical property, we give an alternative proof of the following result in a purely symplectic flavor.

\begin{prop}[=Proposition \ref{prop:ech}]
Assume $(X,\omega)$ has a toric action with moment polygon $\Omega$. Let $X_{\Omega}$ be the convex toric domain in $\CC^2$ associated to $\Omega$. Then $f_k(X,\omega)=c_k^{\rm{ECH}}(X_{\Omega})$.
\end{prop}

Finally we remark that, while not all $c_1$-positive symplectic forms on a rational manifold admit a toric structure, it's shown in \cite{ATFLCY} that they all admit an almost toric structure introduced by \cite{symington}. An almost toric fibration naturally gives a log Calabi-Yau divisor in a rational manifold, which is a generalization of the moment map boundary divisor in the toric situation. We hope that the tropical property established in this paper could help establish a more general correspondence between the tamed capacities of closed manifolds and the capacities defined for open manifolds coming from the divisor complement, under some suitable assumptions (like \cite[conjecture 1.13]{rationalobs}, assuming trivial topology).

\section{Subtlety with $SW,Gr,Gr'$ when $b_2^+=1$}\label{section:SWGrGr'}

 We briefly recall the minimum knowledge needed in this paper about the story of Seiberg-Witten, Gromov-Taubes and McDuff's modified version of Gromov-Taubes invariants. We refer to \cite{Salamonbook}, \cite{Taubes1} and \cite{McDufflecture} for more details. 

In this section, $X$ is always assumed to be a closed oriented $4$-manifold satisfying $b_2^+> 0$ and $H_1(X;\ZZ)=0$ (for simplicity), together with a fixed orientation on $H^0(X;\RR)$, $H^1(X;\RR)$ and $H^2_+(X;\RR)$. The Seiberg-Witten invariant is defined on the set of Spin$^c$ structures $\mathcal{S}_X$. There is a rank $2$ complex vector bundle $\mathcal{L}$, called the positive spinor bundle, associated to each element in $\mathcal{S}_X$. We also need its determinant line bundle $det(\mathcal{L})$. Now choose a Riemannian metric $g$ and a self-dual two form $\mu$. One can then write down the Seiberg-Witten equation and study the moduli space $\mathcal{M}(X,\mathcal{L})$ of the solutions which contain the pairs of a connection on $det(\mathcal{L})$ and a section of $\mathcal{L}$. For generic choice of $(g,\mu)$, $\mathcal{M}(X,\mathcal{L})$ is a compact oriented manifold of real dimension $2d=\frac{c_1(\mathcal{L})^2-2\chi(X)-3\tau(X)}{4}$, over which there is a principal $S^1$-bundle given by gauge transformation. Then, associated to previous datum, there will be an integer defined by integrating $d$-th power of the Euler class of the principal $S^1$-bundle on $\mathcal{M}(X,\mathcal{L})$. When $b_2^+>1$, this number doesn't depend on the generic choices and we thus obtain a well-defined map $$SW:\mathcal{S}_X\rightarrow \ZZ.$$

When $b_2^+=1$, the choice of $(g,\mu)$ matters. Introduce the discriminant $$\Delta_{\mathcal{L}}(g,\mu):=\int_X(c_1(\mathcal{L})-\mu)\wedge\omega_g,$$
where $\omega_g$ is the unique $g$-harmonic self-dual two form in the positive forward cone (given by the orientation on $H^2_+(X;\RR)$) satisfying the normalized condition $\int_X\omega_g^2=1$. Depending on whether $\Delta_{\mathcal{L}}(g,\mu)$ is positive or negative, there will be positive and negative chambers in the space of pairs $(g,\mu)$. Therefore we will obtain two well-defined maps
$$SW_{\pm}:\mathcal{S}_X\rightarrow \ZZ.$$

When the manifold $X$ comes up with a symplectic structure $\omega$ whose canonical class is $K$, there is a distinguished element in $\mathcal{S}_X$ which gives $\mathcal{L}_{K^{-1}}$ with $c_1(\mathcal{L}_{K^{-1}})=-K$. The symplectic form $\omega$ also naturally determines an orientation of $H_+^2(X;\RR)$ (and thus the positive forward cone), which we require to coincide with the one we fixed at the beginning. Moreover, given a homology class $A\in H_2(X;\ZZ)$ and its associated complex line bundle $\mathcal{L}_A$, one will have another spin$^c$ structure whose spinor bundle is given by $\mathcal{L}_{K^{-1}}\otimes \mathcal{L}_A$ with $c_1(\mathcal{L}_{K^{-1}}\otimes \mathcal{L}_A)=-K+2A$. It turns out this correspondence between $\mathcal{S}_X$ and $H_2(X;\ZZ)$ is bijective. So it makes sense to use the following two maps
$$SW_{\omega,\pm}:H_2(X;\ZZ)\rightarrow \ZZ.$$

Now we introduce the index of a class $A\in H_2(X;\ZZ)$ as 
$$ind(A):=A^2-K\cdot A.$$
This number coincides with the real dimension of the moduli space $\mathcal{M}(X,\mathcal{L}_{K^{-1}}\otimes \mathcal{L}_A)$. If the index is non-negative, the wall crossing formula in \cite{LiLiuWallcrossing} tells us
$$|SW_{\omega,+}(A)-SW_{\omega,-}(A)|=1.$$

If $g$ is a metric of positive scalar curvature, then \cite{Witten} tells us the moduli space is empty when using the pair $(g,0)$. When $X=\CC\PP^2\#n\overline{\CC\PP}^2$, there is a classical result by Hitchin:
\begin{lemma}[\cite{Hitchin}]\label{lem:Hitchin}
    If $\{H,E_1,\cdots,E_n\}$ is the standard basis of $H_2(X;\ZZ)$, then there always exists a K\"{a}hler structure $(J,\omega_g,g)$ such that $g$ has positive scalar curvature and $[\omega_g]$ is arbitrarily close to the Poincar\'e dual of $H$.
\end{lemma}

Note that for the K\"{a}hler triple $(J,\omega_g,g)$, $\omega_g$ is automatically $g$-harmonic and self-dual. As a result, we get the following criterion which will be used later. Such a criterion for non-vanishing of the invariant has been used for a long time, see remark \ref{rmk:compare} for the comparison.
\begin{lemma}\label{lem:criterion}
    Let $(X,\omega)$ be a symplectic rational manifold with standard canonical class $K_0=-3H+E_1+\cdots +E_n$. If $A\in H_2(X;\ZZ)$ satisfies $A\cdot H\geq 0$ and $ind(A)\geq 0$, then $SW_{\omega,-}(A)\neq 0$.
\end{lemma}
\begin{proof}
    Note that for the K\"{a}hler positive scalar curvature metric $g$ with K\"{a}hler form $\omega_g$ (this is not the original symplectic form $\omega$), we have  $$\Delta_{\mathcal{L}_{K_0^{-1}}\otimes\mathcal{L}_A}(g,0)=2\omega_g(A)-K_0\cdot [\omega_g]. $$
    If we take $[\omega_g]$ to be very close to the Poincar\'e dual of $H$, the discriminant will be very close to $2A\cdot H+3>0$. This shows $(g,0)$ is in the positive chamber and we thus have $SW_{\omega,+}(A)=0$ and $SW_{\omega,-}(A)\neq 0$.
\end{proof}

We point out a subtle point here. A similar result in \cite[proposition 2.5]{rationalobs} is proved by using the result of \cite{GromovLawsonpsc} instead of \cite{Hitchin}. However, \cite{GromovLawsonpsc} only offers a Riemannian metric $g$ with positive scalar curvature on rational manifolds. It's not known whether $(g,0)$ lies in the positive or negative chamber since one has no knowledge about the class of the harmonic self-dual $2$-form. But \cite{Hitchin} gives a K\"{a}hler metric with positive scalar curvature whose K\"{a}hler form can be used to determine the chamber as above.

Now we review Taubes' program on counting pseudo-holomorphic curves on symplectic $4$-manifold $(X,\omega)$. Let $A\in H_2(X;\ZZ)$ be a non-zero class with $ind(A)\geq 0$.
By choosing a generic compatible almost complex structure $J$ and $\frac{ind(A)}{2}$ generic points on $X$, the Gromov-Taubes invariant $Gr_{\omega}(A)$ is defined by a delicate weighted counting of the moduli space $\mathcal{H}_J(A)$ consisting of $\{(C_i,m_i)\}$ such that
\begin{itemize}
    \item $C_i$'s are disjoint embedded connected $J$-holomorphic submanifolds, $m_i$'s are positive integer numbers, $m_i=1$ unless $C_i$ is a torus with trivial normal bundle.
    \item $\sum m_i[C_i]=A$, $ind([C_i])\geq 0$ for all $i$, each $C_i$ passes through $\frac{ind([C_i])}{2}$ chosen points. 
\end{itemize}

 An immediate observation by adjunction formula and index constraint is that when $[C_i]^2<0$, then $C_i$ must be an exceptional sphere. Now if one further defines $Gr_{\omega}$ to be $1$ for the zero class and $0$ for the classes with negative index, there will be a well-defined map
 $$Gr_{\omega}:H_2(X;\ZZ)\rightarrow \ZZ.$$

 One basic fact is that $Gr_{\omega}$ is invariant under the deformation of the symplectic structures. If $X$ is a rational manifold, by \cite[Theorem D]{LiLiuruled}, any two symplectic forms with the same canonical class are deformation equivalent. By \cite[Theorem 1]{LiLiu01}, up to a diffeomorphism we may always assume the symplectic form on $X$ has the standard canonical class $K_0=-3H+E_1+\cdots+E_n$.  Therefore, later when we only consider the symplectic forms with canonical class $K_0$, we will just write $Gr$ instead of $Gr_{\omega}$. 

By a series of work \cite{Taubes1,Taubes2,Taubes3,Taubes4}, if $b_2^+>1$ one can safely speak
$$SW_{\omega}=Gr_{\omega}:H_2(X;\ZZ)\rightarrow \ZZ.$$

If $b_2^+=1$, one wishes to have $SW_{\omega,-}=Gr_{\omega}$. However, as explained in \cite[Example 6]{McDufflecture}, for $X=\CC\PP^2\#\overline{\CC\PP}^2$ and $A=H+2E$, the Taubes' moduli space $\mathcal{H}_J(A)$ must be empty since the multiply covered exceptional spheres are not allowed, which gives $Gr_{\omega}(A)=0$. On the other hand, by lemma \ref{lem:criterion} we see that $SW_{\omega,-}(A)\neq 0$. To remedy this, McDuff introduced a modified version of Gromov-Taubes invariant $Gr_{\omega}'$ for $b_2^+=1$ which we now recall.

Introduce the following sets for $(X,\omega)$ with canonical class $K$,
 $$\mathcal{E}_X:=\{E\in H_2(X;\ZZ)\,|\,E\text{ can be represented by smooth exceptional sphere}\},$$
$$\mathcal{E}_{\omega}:=\{E\in H_2(X;\ZZ)\,|\,E\text{ can be represented by symplectic exceptional sphere}\},$$ 
 $$\mathcal{E}_{K}:=\{E\in\mathcal{E}_X\,|\,E\cdot K=-1\}.$$A useful fact from \cite{LiLiuWallcrossing} is that $\mathcal{E}_{\omega}=\mathcal{E}_{K}$ so that we will just denote them by $\mathcal{E}$. Now suppose $A\in H_2(X;\ZZ)$ is a non-zero class with $ind(A)\geq 0$, we will look at the following set
 $$\mathcal{E}_A:=\{E\in\mathcal{E}\,|\,E\cdot A\leq -2\}.$$

Note that there is no guarantee this set must be finite. For example, one can take $X=\CC\PP^2\#n\overline{\CC\PP}^2$ with $n>9$ and $A=K-H$. Note that $ind(A)=4$ and the condition $n>9$ makes $\mathcal{E}$ infinite. Moreover, other than $E_1,\cdots,E_n$, all the classes in $\mathcal{E}$ must have positive pairings with $H$. This means $\mathcal{E}_A$ must be infinite. So we need to impose some condition to make $\mathcal{E}_A$ finite so that $Gr'$ could be defined.

\begin{definition}
    We say $Gr'(A)$ {\bf can be defined} if the intersection number of any two distinct elements in $\mathcal{E}_A$ is $0$.
\end{definition}

Note that this naturally makes $\#\mathcal{E}_A\leq b_2^{-}<\infty$. Now we outline the process how to define $Gr'(A)$ under the assumption of the above definition. Firstly, introduce the modified index
$$ind'(A):=ind(A)+\sum_{E\in\mathcal{E}_A}(|E\cdot A|^2-|E\cdot A|),$$choose a generic almost complex structure $J$ and $\frac{ind'(A)}{2}$ generic points on $X$ and then consider the moduli space $\mathcal{H}'_J(A)$ consisting of $\{(C_i,m_i)\}$ such that 
\begin{itemize}
      \item $C_i$'s are disjoint embedded connected $J$-holomorphic submanifolds, $m_i$'s are positive integer numbers, $m_i=1$ unless $C_i$ is a torus with trivial normal bundle or an exceptional sphere.
    \item $\sum m_i[C_i]=A$, $ind'([C_i])\geq 0$ for all $i$, each $C_i$ passes through $\frac{ind'([C_i])}{2}$ chosen points. 
\end{itemize}
Then we define $Gr_{\omega}'(A)$ to be an appropriate count of this $\mathcal{H}'_J(A)$. Note that when $\mathcal{E}_A=\varnothing$, one has $Gr'_{\omega}(A)=Gr_{\omega}(A)$. A typical example showing how this works that one should keep in mind is the previous class $A=H+2E$. The number $Gr_{\omega}'(A)$ actually counts the configuration of a rational curve in class $H$ passing two points and a multiplicity $2$ exceptional sphere in class $E$, which gives $Gr_{\omega}'(A)=Gr_{\omega}(H)=1$. Finally, when $ind(A)< 0$ or $A=0$ we still let $Gr_{\omega}'(A)=Gr_{\omega}(A)$. The main result of \cite{LLSWGr} says
$$SW_{\omega,-}=Gr'_{\omega}:\{A\in H_2(X;\ZZ)\,|\,Gr_{\omega}'(A)\text{ can be defined}\}\rightarrow \ZZ.$$

This identification can be extended over the entire $H_2(X;\ZZ)$ under suitable convention, see remark \ref{rmk:convention}.

\section{Equivalent definitions of tamed capacities for rational manifolds}

\subsection{$f_k(X,\omega,J)$ for a general tamed triple}\label{section:almostcpx}
 In order to study Donaldson's tame-to-compatible question, \cite{comparetamecompatible} introduce the spaces $H^{\pm}_{J}(X)\subseteq H^2(X;\RR)$ for closed almost complex manifolds, where $H^+_J(X)$ is the space where almost K\"{a}hler (compatible) cone live, and $H^-_J(X)$ is a measure of the difference between tamed and compatible cone. Since $J$ acts on the space of real $2$-forms on $X$ by $\alpha(\cdot,\cdot)\mapsto\alpha(J\cdot,J\cdot)$, we have the space of $J$-invariant (resp. anti-invariant) $2$-forms $\Omega_J^+$ (resp. $\Omega_J^-$). Let $\mathcal{Z}^2$ be the space of closed $2$-forms and $\mathcal{Z}_J^{\pm}=\mathcal{Z}^2\cap \Omega_J^{\pm}$. We then define
\[H^{\pm}_J(X):=\{a\in H^2(X;\RR)\,|\,\exists\, \alpha\in \mathcal{Z}_J^{\pm}, a=[\alpha]\},\]

For complex-valued forms, although $\bar{\partial}^2\neq 0$ in the almost complex settings, the bidegree decomposition of forms always exists. So we can talk about $(p,q)$-type complex-valued forms. Define the analogue of Dolbeault cohomology as
\[H_J^{p,q}(X):=\{a\in H^{p+q}(X;\CC)\,|\,\exists (p,q)\text{-type complex-valued form } \alpha, a=[\alpha]\}.\]

\begin{prop}[\cite{comparetamecompatible,DTZ10}]\label{prop:facts}
Let $(X,J)$ be a closed almost complex $4$-manifold. 
    \begin{itemize}
        \item $H^2(X;\RR)=H_J^+(X)\oplus H_J^-(X).$
        \item $H^{2,0}_J(X)\cap H^{1,1}_J(X)\cap H^{0,2}_J(X)=0.$
        \item $H^2(X;\CC)=H^{2,0}_J(X)\oplus H^{1,1}_J(X)\oplus H^{0,2}_J(X)$ iff $J$ is integrable or $H_J^-=0.$
        \item $H_J^+(X)=H_J^{1,1}(X)\cap H^2(X;\RR).$
        \item If $J$ is integrable, $H_J^-(X)=(H_J^{2,0}(X)+H_J^{0,2}(X))\cap H^2(X;\RR).$
        \item If $J$ is not integrable, $H_J^{2,0}(X)+H_J^{0,2}(X)=0.$
        \item If $J$ is integrable and $p+q=2$, by viewing Dolbeault cohomology groups as subspaces of $H^2(X;\CC)$ using the weight $2$ formal Hodge decomposition coming from Fr\"{o}hlicher spectral sequence, $H_J^{p,q}(X)=H_{\bar{\partial}}^{p,q}(X)$.
        \item If $J$ is tamed by some symplectic form, $\text{dim}(H_J^+(X))\geq b_2^-+1,\text{dim}(H_J^-(X))\leq b_2^+-1.$
    \end{itemize}
\end{prop}

To define tamed capacities for any tamed triple $(X,\omega,J)$, we need more notions from almost complex geometry. We refer to \cite{Taubes5,almostkahler,nefclass,Weiyi21} for more details.

\begin{definition}
    An {\bf irreducible $J$-holomorphic subvariety} is a closed subset $C\subset M$ such that
    \begin{itemize}
        \item its $2$-dimensional Hausdorff measure is finite and non-zero;
        \item it has no isolated points;
        \item away from finitely many singular points, $C$ is a smooth submanifold with $J$-invariant tangent space.
    \end{itemize}
    A {\bf $J$-holomorphic subvariety} $\Theta$ is a finite set of pairs $\{(C_i,m_i)\}$ such that each $C_i$ is an irreducible $J$-holomorphic subvariety, $m_i$ is a positive integer and $C_i\neq C_j$ for $i\neq j$.
\end{definition}

Every irreducible $J$-holomorphic subvariety is the image of a $J$-holomorphic map $\phi:\Sigma\rightarrow X$ from a Riemann surface $\Sigma$. Therefore we can define the class of the subvariety $[C]$ as $\phi_*([\Sigma])$. For a class $A\in H_2(X;\ZZ)$, we let the moduli space $\mathcal{M}_A$ be the space of subvarieties $\Theta=\{(C_i,m_i)\}$ such that $[\Theta]:=\sum m_i[C_i]=A$, which can be naturally equipped with a topology in the Gromov-Hausdorff sense.

\begin{definition}
    A class $A\in H_{2}(X;\ZZ)$ is called {\bf $J$-effective} if $\mathcal{M}_A$ is non-empty; is called {\bf $J$-nef} if its intersection pairings with all $J$-effective classes are non-negative.
\end{definition}

\begin{definition}
    For tamed triple $(X,\omega,J)$, define the {\bf $k$-th tamed capacity} $$f_k(X,\omega,J):=\inf \{\omega(A)\,|\,A\in H_2(X;\ZZ)\cap\text{PD}(H_J^+(X)),A\text{ is } J\text{-nef}, ind(A)\geq 2k\}.$$
\end{definition}

\begin{rmk}
      For K\"{a}hler triple $(X,\omega,J)$,  $H_2(X;\ZZ)\cap\text{PD}(H_J^+(X))$ is exactly the Poincar\'{e} dual of the N\'{e}ron-Severi group $H^{1,1}(X,\ZZ)$. By the last fact in proposition \ref{prop:facts}, for manifolds with $b_2^+=1$ like rational manifolds, there is no difference between $H_2(X;\ZZ)$ and $H_2(X;\ZZ)\cap\text{PD}(H_J^+(X))$. So in the later computations of $f_k$'s for rational manifolds, we can just do optimizations on the entire $H_2(X;\ZZ)$.
\end{rmk}

 When $Y$ is a smooth rational surface, the ample divisor $A$ gives a K\"{a}hler form $\omega_A$. Recall that the definition of algebraic capacities can be read as
$$c_k^{\rm{alg}}(Y,A):=\inf_{D\in \text{Nef}(Y)_\ZZ}\{\omega_A(D)\,|\,ind(D)\geq 2k\}.$$
Then if we denote the complex structure by $J$, we naturally have $$c_k^{\rm{alg}}(Y,A)=f_k(Y,\omega_A,J).$$
Note that such formulations of capacities in terms of area versus index are also fundamental to Hutchings' ECH capacities introduced in \cite{Hutquantitative}.

\subsection{A hierarchy of subsets $U_i$'s in $\{ind\geq 2k\}$ for a rational manifold}\label{section:hierarchy}
In this section, we will focus on rational manifolds $\CC\PP\#n\overline{\CC\PP^2}$.
\begin{lemma}\label{lem:eff=sw}
   For any tamed triple $(X,\omega,J)$ with $X$ being a rational manifold, $$\{\text{J-effective classes with } ind\geq 0\}=\{SW_{\omega,-}\neq 0\}.$$
\end{lemma}

\begin{proof}
    The inclusion $\supseteq$ essentially comes from Taubes' work (see also section 2.4 in \cite{Weiyi21}). For $\subseteq$, firstly note that $H$ must be $J$-effective since $Gr(H)\neq 0$. Then use positivity of intersection and apply lemma \ref{lem:criterion}.
\end{proof}

\begin{lemma}\label{lem:Gr'defined}
    For tamed triple $(X,\omega,J)$ where $X$ is a rational manifold with standard basis $\{H,E_1,\cdots,E_n\}$ and canonical class $K_0=-3H+E_1+\cdots+E_n$, if the class $A\in H_2(X;\ZZ)$ with $ind(A)\geq 0$ satisfies $A\cdot H>0$ then $Gr'(A)$ can be defined.
\end{lemma}
\begin{proof}
Note that $A\cdot H>0$ gives $SW_{\omega,-}(A)\neq 0$ by lemma \ref{lem:criterion}. By lemma \ref{lem:eff=sw}, $A$ can also be represented by some $J$-holomorphic subvariety $\Theta=\{(C_i,m_i)\}$. Assume $[e_1]$ and $[e_2]$ are two classes in $\mathcal{E}_A$ such that $[e_1]\cdot [e_2]\geq 1$. Since $[e_1]\cdot A\leq -2$, by positivity of intersection, there must be some $[C_i]$ equal to $[e_1]$. For the same reason there is some $[C_j]$ equal to $[e_2]$. But then one will have the contradiction 
$$-4\geq ([e_1]+[e_2])\cdot A\geq ([e_1]+[e_2])\cdot(m_i[C_i]+m_j([C_j]))\geq (-m_i+m_j)+(m_i-m_j)=0.$$
\end{proof}

\begin{rmk}\label{rmk:convention}
    Another convention is to extend the definition of $Gr'$ to be $0$ for those classes that can not be defined in our sense. From the proof of lemma \ref{lem:Gr'defined} we see that if $Gr'(A)$ can not be defined, then we must have $SW_{\omega,-}(A)=0$. In this sense one can say $Gr'=SW_{\omega,-}$ holds for all homology classes.
\end{rmk}

\begin{rmk}\label{rmk:compare}
    \cite[lemma 2.1.5]{MO} and \cite[proposition 4.3]{Danconcave} give two sufficient conditions for a class having non-zero Gromov invariant. Their Gromov invariants, although written as $Gr$, should be understood as $Gr'$ in the sense of remark \ref{rmk:convention}. Keeping this in mind, we can compare them with lemma \ref{lem:criterion}:
    \begin{itemize}
        \item \cite{MO}'s condition: $A^2\geq 0$, $\omega(A)>0$ and $ind(A)\geq 0$. Note that by light cone lemma, this implies our condition $A\cdot H>0$. There exists classes like $H+E_1+E_2$ with non-zero invariant, which satisfy the condition in lemma \ref{lem:criterion} but doesn't satisfy this criterion. 
        \item \cite{Danconcave}'s condition: $\omega(PD(K_0)-A)<0$ and $ind(A)\geq 0$. Since $Gr'$ doesn't depend on the deformation class of $\omega$, we can take $\omega$ to be the symplectic form coming from small blowups so that $PD([\omega])$ is very close to $H$. In this case, $\omega(PD(K_0))$ is very close to $-3$ and we have $A\cdot H\geq -3$. By index constraint, it's easy to exclude $A\cdot H=-3,-2,-1$. Therefore it coincides with our criterion in lemma \ref{lem:criterion}.
    \end{itemize}
\end{rmk}

Now we consider a tamed triple $(X,\omega,J)$ with a standard basis $\{H,E_1,\cdots,E_n\}$ and canonical class $K_0=-3H+E_1+\cdots+E_n$ for the rational manifold $X$. We introduce the following subsets $U_i:=\{A\,|\,ind(A)\geq 2k,\star\}$\footnote{For simplicity we don't use the notations like $U_i^{(k)}$, although these $U_i$'s do depend on $k$.} in $H_2(X;\ZZ)$ where $\star$ is the $i$-th requirement in the followings:
\begin{enumerate}
    \item $J$-nef;
    \item $A\cdot H>0$, $A^2\geq 0$ and $A\cdot E\geq 0$ for all $E\in\mathcal{E}_{K_0}$;
    \item $Gr(A)\neq 0$;
    \item $Gr'(A)$ can be defined and $Gr'(A)\neq 0$;
    \item $A\cdot H>0$;
    \item $SW_{\omega,-}(A)\neq 0$;
    \item $J$-effective.
\end{enumerate}

It turns out, except for $U_1$, all these subsets will not change as $(\omega,J)$ deforms due to the following result.

\begin{prop}\label{prop:hierarchy}
With the settings as above, we have\footnote{Here the strict inclusions should be understood as `not always equal'. That is, there exists some $(X,\omega,J)$ such that $U_1,U_2,U_3,U_4$ are not equal.} $$U_1\subsetneq U_2\subsetneq U_3\subsetneq U_4= U_5=U_6=U_7.$$
\end{prop}

\begin{proof}
First of all, observe that if $J$ comes from the complex blowup of three points on the same line in $\CC\PP^2$, then there is a curve in class $H-E_1-E_2-E_3$ and thus $2H-E_1-E_2-E_3\in U_2\setminus U_1$; $H+E_1+E_2\in U_3\setminus U_2$; $H+2E_1\in U_4\setminus U_3$. Then we explain each inclusion relation one by one:
    \begin{itemize}
        \item ($U_1\subseteq U_2$) If $A$ is $J$-nef, then by positivity of intersection, we have $A\cdot E\geq 0$ for all $E\in\mathcal{E}_{K_0}$ and $A\cdot H\geq 0$. The condition $ind(A)\geq 2k\geq 2$ also forces $A\cdot H>0$. Therefore by the discussion of section \ref{section:SWGrGr'} we know that $Gr(A)=Gr'(A)=SW_{\omega,-}(A)\neq 0$. To see $A^2\geq 0$, note that $Gr(A)\neq 0$ means for some generic almost complex structure (maybe not the same as $J$), there is some subvariety $\{(C_i,m_i)\}$ in Taubes' moduli space representing the class $A$. All $C_i$'s are disjoint and the only components with negative self-intersection must be embedded exceptional spheres whose classes live in $\mathcal{E}_{K_0}$. However, since we already know the class $A$ pairs non-negatively with any element in $\mathcal{E}_{K_0}$, there couldn't be any $C_i$ with negative self-intersection. Therefore $A^2=\sum_im_i^2[C_i]^2\geq 0$. 
        \item ($U_2\subseteq U_3$) The conditions $A\cdot E\geq 0$ for all $E\in\mathcal{E}_{K_0}$ and $A\cdot H>0$ guarantee $Gr(A)=Gr'(A)=SW_{\omega,-}(A)\neq 0$.
        \item ($U_3\subseteq U_4$) The element $\{(C_i,m_i)\}$ in Taubes' moduli space has the property that if $C_i^2<0$ then $C_i^2=-1$ and $m_i=1$. By positivity of intersection, $A\cdot E\geq -1$ for all $E\in\mathcal{E}_{K_0}$. So $Gr'(A)=Gr(A)\neq 0$.
        \item ($U_4\subseteq U_5$) Positivity of intersection.
        \item ($U_5\subseteq U_6$) By lemma \ref{lem:criterion}.
        \item ($U_6=U_7$) By lemma \ref{lem:eff=sw}.
        \item ($U_6\subseteq U_4$) By lemma \ref{lem:Gr'defined} and the identification in \cite{LLSWGr}.
    \end{itemize}
\end{proof}


\subsection{Proof of equivalence}
We fix a positive integer $k$ and keep the same settings as the last section. The aim of this section is to show:

\begin{prop}\label{prop:equivalentdef}
We have equal values 
 $$\inf_{A\in U_1}\omega(A)=\inf_{A\in U_2}\omega(A)=\cdots=\inf_{A\in U_7}\omega(A).$$ 
\end{prop}

Firstly note that if there is a sequence of $J$-holomorphic subvarieties $\Theta_i$ whose $\omega$-areas converge to the infimum over $U_7$, then by Gromov compactness theorem, after passing to a subsequence it will have a limit $\Theta$. Note that $[\Theta_i]=[\Theta]$ for large $i$ so that $[\Theta]\in U_7$. To summarize, we then get

\begin{lemma}\label{lem:min}
    There always exists $A_{min}\in U_7$ such that $\omega (A_{min})$ attains the infimum over $U_7$. 
\end{lemma}

\begin{proof}[Proof of proposition \ref{prop:equivalentdef}]
    By the above lemma, it suffices to show $A_{min}$ is $J$-nef. Since $A_{min}\in U_7$ has effective representation $\{(C_i,m_i)\}$, if $e$ is an irreducible curve such that $[e]\cdot A_{min}\leq -1$, $e$ must be equal to some $C_i$ by positivity of intersection. Now let's consider the class $A_{min}-[e]$, which is also effective. Its index satisfies
    $$(A_{min}-[e])^2+c_1(A_{min}-[e])=ind(A_{min})+[e]^2+K_0\cdot [e]-2A_{min}\cdot [e]\geq ind(A_{min})$$ 
    by assumption and adjunction inequality. But then $A_{min}-[e]$ will have smaller area which is a contradiction.
\end{proof}

Such an argument of peeling off process also appears in the proof of \cite[theorem 3.4]{rationalobs}, which uses a technical result from algebraic geometry. However, working in the almost complex category rather than algebraic category, the proof can be as simple as above.

Proposition \ref{prop:equivalentdef} gives us seven different perspectives for doing optimizations in order to define the tamed capacities for rational manifolds. We will use $U_5$ in the next section to prove our main theorem. In particular, we immediately see 
\begin{corollary}\label{cor:indepofj}
    When $X$ is a rational manifold, the $k$-th tamed capacity $f_k(X,\omega,J)$ only depends on the cohomology class $[\omega]$.
\end{corollary}

Therefore, we can omit the almost complex structure $J$ and just write $f_k(X,[\omega])$ from now on if $X$ is a rational manifold.

\begin{rmk}\label{rmk:exampleT4}
One can expect the $J$-independence still holds for ruled manifolds since it's shown in \cite{Sungruled} that ruled manifolds have the similar curvature property as lemma \ref{lem:Hitchin}. However, this is not the case for other manifolds. A quick example is by taking $X$ to be $K3$ or $T^4$ endowed with a K\"{a}hler pair $(\omega,J)$ from a polarization. One can always deform the complex structure into another $J'$ whose Picard number is $0$ by choosing a generic point in the period domain. Since the tame condition is an open condition, $(\omega,J')$ is still a tamed pair if $J'$ is sufficiently close to $J$. Then one can see $f_k(X,\omega,J)$ is always finite but $f_k(X,\omega,J')=\infty$ since the infimum is taken over an empty set. See more discussions in section \ref{section:comparison}.
\end{rmk}

\section{Tropical property}
From now on, $X$ will always be a rational manifold. We want to treat $f_k(X,[\omega])$ as a function with variable $[\omega]\in H^2(X;\RR)$. The first thing is to figure out the suitable domain for this function.

\subsection{Various kinds of symplectic cones}\label{section:symplecticcone}
Assume $X=\CC\PP^2\#n\overline{\CC\PP}^2$ with standard basis $\{H,E_1,\cdots,E_n\}\subset H_2(X;\ZZ)$. We can naturally identify $\RR^{n+1}$ with $H^2(X;\RR)$ using this basis by associating $(x_0,x_1,\cdots,x_n)$ to $x_0PD(H)-\sum_{i=1}^nx_iPD(E_i)$. Define the {\bf symplectic cone} to be
$$\mathcal{C}:=\{[\omega]\,|\,\omega\text{ is a symplectic form on }X\}\subset \RR^{n+1}.$$
By \cite{LiLiu01}, up to a diffeomorphism, any symplectic form on $X$ has the standard canonical class $PD(-3H+E_1+\cdots+E_n)$, which we denote by $K_0$ so as to distinguish it from other possible symplectic canonical classes. This leads to the definition of the {\bf symplectic $K_0$-cone}
$$\mathcal{C}_{K_0}:=\{[\omega]\,|\,[\omega]\in\mathcal{C},K_{\omega}=K_0\}\subset \RR^{n+1}.$$

Let $D_{K_0}\subseteq Aut(H_2(X;\RR))$ be the group of homological actions by diffeomorphisms preserving the canonical class $K_0$. Note that $D_{K_0}$ naturally acts on $\mathcal{C}_{K_0}$. This action has a fundamental domain $\mathcal{P}$, which is called the {\bf reduced cone}, given by
\begin{itemize}
    \item $0<x_1<x_0$ when $n=1$;
    \item $0<x_2\leq x_1$, $x_1+x_2<x_0$ when $n=2$;
    \item $0<x_n\leq\cdots\leq x_1$, $x_1+x_2+x_3\leq x_0$ and $\sum_{i=1}^nx_i^2<x_0^2$ when $n\geq 3$.
\end{itemize}
Note that the condition $\sum_{i=1}^nx_i^2<x_0^2$ is redundant if $3\leq n\leq 8$ by Cauchy inequality. For more details, we refer to \cite{ALLP-stability} and the reference therein. The upshot is the following result due to Karshon and Kessler.

\begin{theorem}[\cite{KK17}]
   There is a natural bijection
$$\{\text{symplectic forms on }X\}/\text{symplectomorphisms}\xleftrightarrow[]{1-1} \mathcal{P}.$$ 
\end{theorem}

 If we want to also ignore the scalar multiplications, we can consider the {\bf normalized reduced region}
$$\tilde{\mathcal{P}}:=\{[\omega]\in\mathcal{P}\,|\,\omega(H)=1\}.$$

When $n\leq 9$, $\tilde{\mathcal{P}}$ has the shape of a polytope with some parts of the boundary points removed; when $n\geq 10$, $\tilde{\mathcal{P}}$ looks like a polytope with a chopped corner determined by a quadratic equation (see figure \ref{fig:cone2}).

Finally, we introduce the {\bf $c_1$-positive cone} 
$$\mathcal{P}^{c_1>0}:=\{[\omega]\in\mathcal{P}\,|\,\omega(3H-E_1-\cdots -E_n)>0\},$$
and its normalized version {\bf $c_1$-positive region}
$$\tilde{\mathcal{P}}^{c_1>0}:=\{[\omega]\in\mathcal{P}^{c_1>0}\,|\,\omega(H)=1\}.$$
Note that when $n\leq 9$, $\tilde{\mathcal{P}}^{c_1>0}=\tilde{\mathcal{P}}$; when $n\geq 10$, $\tilde{\mathcal{P}}^{c_1>0}\subsetneq\tilde{\mathcal{P}}$ and a nice property we will rely on is 
\begin{lemma}[\cite{Enumerate} lemma 3.1]
    $\tilde{\mathcal{P}}^{c_1>0}$ is a polyhedral region.
\end{lemma}
We describe the pattern of its vertices for the preparation of next section. To avoid tautology, we drop the first constant coordinate $x_0=1$ and only write $(x_1,\cdots,x_n)$ below. 
\begin{itemize}
    \item $n\leq 9$: $P_1=(0,0,\cdots,0),P_2=(1,0,\cdots,0),P_3=(\frac{1}{2},\frac{1}{2},0\cdots,0)$ and $P_i=(\underbrace{\frac{1}{3},\cdots,\frac{1}{3}}_{i-1},0,\cdots,0)$ for $4\leq i\leq n+1$;
    \item $n\geq 10$: $P_1,\cdots,P_{10}$ as $n\leq 9$, and $Q_{ij}:=t_{ij}P_i+(1-t_{ij})P_{j+1}$ for each $1\leq i\leq 9$, $10\leq j\leq n$, where $t_{ij}$ is the unique number such that the sum of all entries of $Q_{ij}$ is $3$. The precise values for these $t_{ij}$ are spelled out in \cite[section 2.2]{llwtorelli}. But for the purpose of this paper, we only need to remember $0<t_{ij}<1$.
    
\end{itemize}

We prepare the following cartoons (figures \ref{fig:cone1},\ref{fig:cone2}) to help readers figure out the concepts introduced in this section.
\begin{figure}[H]
		\includegraphics*[width=\linewidth]{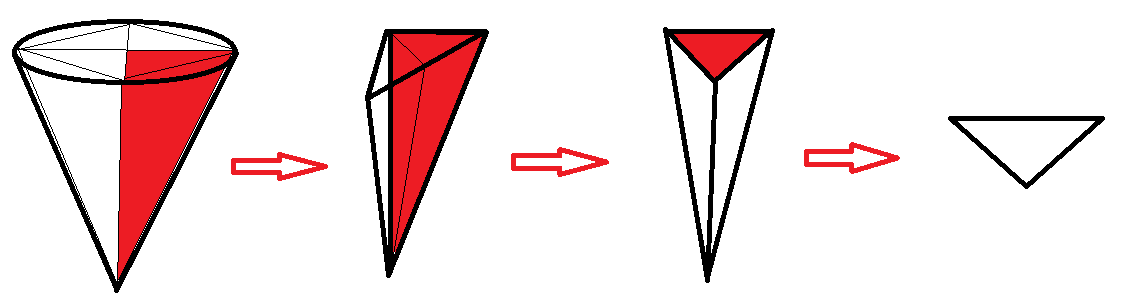}
 
		\caption{From left to right, these are the symplectic cone $\mathcal{C}$, symplectic $K_0$-cone $\mathcal{C}_{K_0}$, reduced cone $\mathcal{P}$ and normalized reduced region $\tilde{\mathcal{P}}$ for $\CC\PP^2\#2\overline{\CC\PP}^2$. Note that $\mathcal{C}$ is the region $\{x_1^2+x_2^2<x_0^2\}\subseteq \RR^3$ with infinitely many walls removed (we only draw $6$ of them, see also \cite[figure 13.4]{MSbook}). These walls correspond to the condition $\omega(E)=0$ for some exceptional class $E$. By adding the restrictions $x_1>0,x_2>0,x_1+x_2<1$ we get $\mathcal{C}_{K_0}$. Now the group $D_{K_0}$ will act on $\mathcal{C}_{K_0}$. In two points blowup case, this action is simply the reflection along the plane $\{x_1=x_2\}$. By further adding the restriction $x_2\leq x_1$, we get a fundamental domain under this action, which is $\mathcal{P}$. Finally, we take a cross section by adding the condition $x_0=1$ to obtain $\tilde{\mathcal{P}}$. \label{fig:cone1}}
	\end{figure}

\begin{figure}[h]
		\centering\includegraphics*[height=5cm, width=8cm]{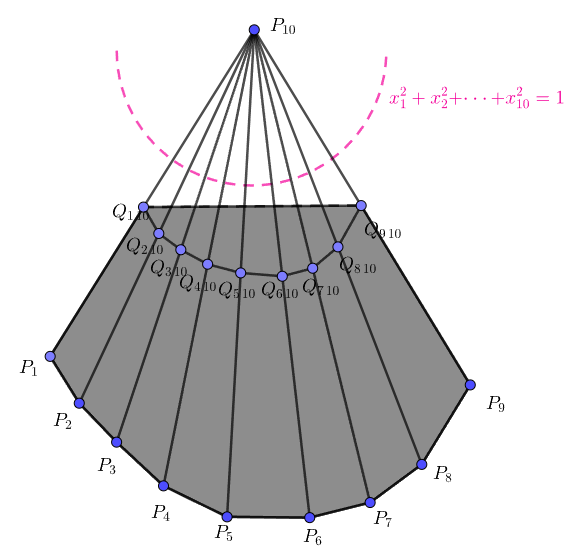}
 
		\caption{The imaginary picture for $\tilde{\mathcal{P}}$ of $\CC\PP^2\#10\overline{\CC\PP}^2$. The $c_1$-positive region $\tilde{\mathcal{P}}^{c_1>0}$ is represented by the shaded region in the figure where we mark all the vertices $P_i$ and $Q_{ij}$. Note that when the number of blowup points is greater than $9$, $\tilde{\mathcal{P}}$ will have part of the boundary defined by a quadratic equation. However, $\tilde{\mathcal{P}}^{c_1>0}$ is still a polytope.  \label{fig:cone2}}
	\end{figure}
 
\subsection{Finiteness}
Now we can view the $k$-th tamed capacity $f_k$ as a function over $\mathcal{C}$ or some of its subsets described in the previous section. The aim of this section is to prove the following:
\begin{theorem}\label{thm:main}
    $f_k$ is a tropical polynomial over the reduced $c_1$-positive cone $\mathcal{P}^{c_1>0}$; when $n\leq 8$, the tropical property still holds over the larger symplectic $K_0$-cone $\mathcal{C}_{K_0}$.
\end{theorem}

By proposition \ref{prop:equivalentdef}, we will consider the optimization problem over the set 
$$U_5=\{A\in H_2(X;\ZZ)\,|\,ind(A)\geq 2k,A\cdot H>0\}$$
which doesn't depend on either $\omega$ or $J$ when we fix the canonical class $K_0$. Firstly, we require the following fact:

\begin{lemma}\label{lemma:U5preserve}
    (1) Any $\phi\in D_{K_0}$ satisfies $\phi(U_5)=U_5$.
    
    (2) Assume $n\leq 8$, then for any $A\in U_5$, the orbit set $\{\phi(A)\,|\,\phi\in D_{K_0}\}\subset U_5$ is finite.\end{lemma}

\begin{proof}
    (1) We only need to prove $\phi(U_5)\subseteq U_5$ and then apply it to $\phi^{-1}$. By \cite[Proposition 4.7]{LiWuLag} (see also \cite[Proposition 1.2.12]{Ellipsoidannals}), $\phi$ is generated by reflections along $E_i-E_j$ and $H-E_1-E_2-E_3$. For $A=aH-\sum_{i=1}^nb_iE_i\in U_5$, the reflections along $E_i-E_j$ obviously will stay in $U_5$ since they merely switch $b_i$ and $b_j$. For reflection along $H-E_1-E_2-E_3$, it will turn $A$ into $$(2a-b_1-b_2-b_3)H-(a-b_2-b_3)E_1-(a-b_1-b_3)E_2-(a-b_1-b_2)E_3-\sum_{i\geq 4} b_iE_i.$$ If $2a-b_1-b_2-b_3\leq 0$, then $b_1^2+b_2^2+b_3^2\geq \frac{4}{3}a^2$ and thus $$ind(A)=a^2+3a-\sum_{i=1}^n(b_i^2+b_i)\leq -\frac{1}{3}a^2+a<2,$$ which is a contradiction. So it also preserves $U_5$.

    (2) Note that $K_0\cdot \phi(A)=K_0\cdot A$. Therefore if $\phi(A)=cH-\sum _{i=1}^nd_iE_i$, we have
    $$ind(\phi(A))=c^2-\sum_{i=1}^nd_i^2-K_0\cdot A\leq c^2-\frac{1}{8}(\sum_{i=1}^nd_i)^2-K_0\cdot A$$
    $$=c^2-\frac{1}{8}(3c+K_0\cdot A)^2-K_0\cdot A=-\frac{1}{8}c^2-\frac{3K_0\cdot A}{4}c-((K_0\cdot A)^2+K_0\cdot A).$$
    We see that for $ind(\phi(A))\geq 2k$, $c$ must be bounded and thus the orbit is finite.
\end{proof}

To prove Theorem \ref{thm:main}, we now introduce a partial order $\geq$ on the set $U_5$. For $A,B\in U_5$ we say $A\geq B$ if $\omega(A)\geq \omega(B)$ for all $[\omega]$ in the $c_1$-positive cone $\mathcal{P}^{c_1>0}$. We say $A\in U_5$ is {\bf minimal} if there doesn't exist any other $B\in U_5$ such that $A\geq B$. A class of the form $aH-\sum_{i=1}^nb_iE_i$ with $a>0,b_1\geq b_2\geq\cdots b_n\geq 0,a\geq b_1+b_2+b_3$ is called {\bf reduced}.

\begin{lemma}\label{lem:stabilize}
 Every chain $A_1\geq A_2\geq\cdots$ in $(U_5,\geq)$ will stabilize. 
 \end{lemma}
\begin{proof}
    Pick finitely many generic points in the interior of $c_1$-positive cone so that the evaluation of a class on those points determines the evaluation on the entire polytope. If the chain doesn't stabilize, there must exist one $[\omega]$ among those points such that $\omega(A_1)>\omega(A_2)>\cdots$.  Then it will contradict Gromov compactness as lemma \ref{lem:min}.
\end{proof}

 \begin{lemma}\label{lem:reduced}
Every minimal element $A=aH-\sum_{i=1}^nb_iE_i\in U_5$ must be reduced and satisfies $3a-\sum_{i=1}^nb_i\geq 1$.
\end{lemma}
\begin{proof}
 We may assume $a>3$ by checking all the cases for $a=1,2,3$ from the index constraint. If $b_i<0$ for some $1\leq i \leq n$, then $A':=A+b_iE_i\in U_5$ will satisfy $A\geq A'$; if $b_i<b_j$ for some $1\leq i<j\leq n$, then $A':=A-(b_j-b_i)(E_i-E_j)\in U_5$ will satisfy $A\geq A'$; if $a<b_1+b_2+b_3$, then $A':=A-(H-E_1-E_2-E_3)\in U_5$ will satisfy $A\geq A'$. Thus $A$ must be reduced. If $3a-\sum_{i=1}^nb_i\leq 0$, then $A':=A-(3H-\sum_{i=1}^n E_i)$ has $ind(A')=ind(A)-2(3a-\sum_{i=1}^nb_i)\geq ind(A)$. This implies $A'\in U_5$ and $A\geq A'$. We thus know $3a-\sum_{i=1}^nb_i\geq 1$.
\end{proof}

\begin{lemma}\label{lem:c1bound}
    For each $n,k$, there exist constants $A_{n,k},C_{n,k}$ such that if $A=aH-\sum_{i=1}^nb_iE_i\in U_5$ with $a>A_{n,k}, 3a-\sum_{i=1}^{\min\{n,9\}}b_i>C_{n,k},3a-\sum_{i=1}^nb_i\geq 1$ and is reduced, then $A\geq kH-kE_1$. In particular, $A$ is not minimal.
\end{lemma}

\begin{proof}
The strategy is to test the values of $A$ and $kH-kE_1$ at all the vertices of the $c_1$-positive polytope $\tilde{\mathcal{P}}^{c_1>0}$. When $n\leq 9$, the $c_1$-positive polytope has vertices $P_1=(0,\cdots,0),P_2=(1,0,\cdots,0),P_3=(\frac{1}{2},\frac{1}{2},0,\cdots,0),P_4=(\frac{1}{3},\frac{1}{3},\frac{1}{3},0,\cdots,0),\cdots, P_{n+1}=(\underbrace{\frac{1}{3},\cdots,\frac{1}{3}}_{n})$. When $n\geq 10$, the additional vertices will be the point on the segment connecting $(\underbrace{\frac{1}{3},\cdots,\frac{1}{3}}_{\geq 10},0,\cdots,0)$ and some $P_j$ satisfying the sum of all entries is $3$. They are of the forms $Q_{jl}:=tP_j+(1-t)(\underbrace{\frac{1}{3},\cdots,\frac{1}{3}}_{l-1\geq 10},0,\cdots,0)$ with $0< t< 1$. The values of $A$ at $P_1,P_2,P_3$ are $a,a-b_1,\frac{1}{2}(a+a-b_1-b_2)\geq \frac{1}{2}a$. So when $a>k$, those values will be larger than the values of $kH-kE_1$ at $P_1,P_2,P_3$ which are $k,0,\frac{k}{2}$.

At the vertices $Q_{2l},Q_{3l}$, the values of $A$ are 
\[t(a-b_1)+\frac{1-t}{3}(3a-\sum_{i=1}^lb_i)\,,\,t(a-\frac{1}{2}b_1-\frac{1}{2}b_2)+\frac{1-t}{3}(3a-\sum_{i=1}^lb_i).\] Since $0<t<1$, both of them would be greater than $\varepsilon\max\{a-b_1,3a-\sum_{i=1}^lb_i\}$ for some positive number $\varepsilon=\min\{t,\frac{1-t}{3}\}$. Now we use the fact:
\begin{itemize}
	\item For any $X,Y>0$, there exists a constant $\alpha_{X,Y}$ such that if $a>\alpha_{X,Y},a-b_1<X$, then $3a-\sum_{i=1}^lb_i>Y$.
\end{itemize} 
This is because by reduced condition on the class $A$, $a-b_1<X$ will imply $b_3\leq \frac{X}{2}$. Then one can see $$3a-\sum_{i=1}^lb_i\geq 2a+(a-b_1-b_2)-(l-2)b_3\geq 2a-(l-2)\frac{X}{2}.$$ Therefore there exists $A_{n,k}$ such that if $a\geq A_{n,k}$ then $\max\{a-b_1,3a-\sum_{i=1}^lb_i\}$ is large enough.

 For all the other vertices, we claim that the values of $A$ must be greater than $a-\frac{1}{3}\sum_{i=1}^{\min\{n,9\}}b_i$. This is obvious for $P_4,P_5,\cdots$. For other $Q_{jl}$'s with $j\neq 2,3$, we can use another fact which is the consequence of Chebyshev's sum inequality and the reduced condition $b_1\geq \cdots\geq b_n$:
\begin{itemize}
    \item $-\sum_{i=1}^l\lambda_ib_i\geq -\frac{1}{3}\sum_{i=1}^9b_i$ if $\sum_{i=1}^l\lambda_i=3$ and $0\leq \lambda_i\leq \frac{1}{3}$ for all $i$.
\end{itemize}
 Finally it suffices to let $C_{n,k}$ be the triple of maximum among the values of $kH-kE_1$ at all the vertices. 
\end{proof}

Now our main result can be converted to the following:
\begin{prop}\label{prop:finite}
    The set of minimal elements in $U_5$ is finite.
\end{prop}

\begin{proof}
 We will show minimal elements $A=aH-\sum_{i=1}^nb_iE_i\in U_5$ have a common upper bound of the coefficient $a$. By Lemma \ref{lem:reduced} and \ref{lem:c1bound} we can further assume $A$ is reduced and satisfies \[3a-\sum_{i=1}^nb_i\geq 1\,,\, 3a-\sum_{i=1}^{\min\{n,9\}}b_i\leq C_{n,k}.\] When $n<9$ we can pretend that we are in the case when $n=9$ with $b_{n+1}=\cdots=b_9=0$. So next we always suppose $n\geq 9$. Note that from the assumptions, \[\sum_{i>9}b_i=(3a-\sum_{i=1}^{9}b_i)-(3a-\sum_{i=1}^{n}b_i)< C_{n,k}.\]
 Denote $3H-\sum_{i=1}^8E_i$ by $B$. We can write \[A=cH-\sum_{i=1}^8d_iE_i+\sum_{i=9}^nb_i(B-E_i),\]
 where $c=a-3\sum_{i\geq 9}b_i, d_i=b_i-\sum_{j\geq 9}b_j$. Since $3c-\sum_{i=1}^8d_i=3a-\sum_{i=1}^nb_i$ we have $1\leq 3c-\sum_{i=1}^8d_i\leq C_{n,k}$. By reduced condition on $A$, \[X_1:=c-d_1-d_2-d_3,X_2:=c-d_4-d_5-d_6,X_3:=c-d_7-d_8\] will satisfy $0\leq X_1\leq X_2\leq X_3$. For any $1\leq i\leq 8$, the inequality $X_1+X_2+X_3\leq C_{n,k}$ would then imply $$d_i\leq d_1\leq d_1+(d_2-d_7)+(d_3-d_8)=X_3-X_1\leq C_{n,k}.$$ On the other hand, we also have $$d_i=(b_i-b_9)-\sum_{j>9}b_j>-C_{n,k}.$$ Thus we have a bound $D_{n,k}$ for all $|d_i|$ and $|c|$. We then compute \[ind(A)=c^2-\sum _{i=1}^8d_i^2+\sum_{9\leq i<j\leq N}2b_ib_j+(1+\sum_{i=9}^n2b_i)(3c-\sum_{i=1}^8d_i)\]
 \[\geq c^2-\sum _{i=1}^8d_i^2+\sum_{i=9}^n2b_i\geq -8D_{n,k}^2+\sum_{i=9}^n2b_i .\]
As a consequence, whenever $\sum_{i=9}^nb_i>4D_{n,k}^2+k$, $A-(B-E_i)$ still has index $\geq 2k$ for any $9\leq i\leq n$. Since we may assume $a>3$, $A-(B-E_i)\in U_5$. One can see that $A\geq A-(B-E_i)$. Therefore for a minimal element $A$, we then have $a=c+3\sum_{i=9}^nb_i\leq D_{n,k}+12D_{n,k}^2+3k$.
\end{proof}

\begin{proof}[Proof of Theorem \ref{thm:main}]
Combining lemma \ref{lem:stabilize} and proposition \ref{prop:finite} we get the tropical property over the $c_1$-positive cone: there exists finitely many $A_1,\cdots,A_m\in U_5$ such that the capacity function is given by 
$$f_k|_{\mathcal{P}^{c_1>0}}([\omega])=\min \{\omega(A_1),\cdots,\omega(A_m)\}.$$

Now for any $[\omega']\in\mathcal{C}_{K_0}$, it is given by $g^*([\omega])$ for some $g^*\in D_{K_0}$\footnote{Here we are sloppy about the notation: we use $D_{K_0}$ to denote also the cohomological actions.} and $[\omega]\in\mathcal{P}^{c_1>0}$. We then have
$$f_k|_{\mathcal{C}_{K_0}}([\omega'])=\min \{\omega'(g_*(A_i))\,|\,g_*\in D_{K_0},1\leq i\leq m\}.$$ Therefore, when $n\leq 8$, by lemma \ref{lemma:U5preserve} the set on the right hand side must be finite and we have the tropical property over $\mathcal{C}_{K_0}$.
\end{proof}

\begin{rmk}
    The number of minimizers $m$ appeared above describes the amount of curve classes which really play a role in obstructing symplectic embeddings. In terms of the upper bound for the $a$ coefficient in the proof of proposition \ref{prop:finite}, say $M$, it's easy to get a `cheap' upper bound for $m$ such as $1^n+2^n+\cdots+M^n$ since minimal elements have all $b_i$ coefficients no larger than $a$ . This estimate and the estimate for the upper bound $M$ are very rude, so it's an interesting question to investigate the more precise values of $m$. 
\end{rmk}

\begin{rmk}\label{rmk:nonc1nef}
    One can not hope to extend the tropical property from $\mathcal{P}^{c_1>0}$ to the larger cone $\mathcal{P}$ removing the $c_1$-positive condition. Consider the following classes for $n=10$:
    \[N_{a}=3H-E_1-\cdots-E_8+\frac{a^2+a}{2}(3H-E_1-\cdots-E_9)-aE_{10}.\] Then one can compute \[ind(N_{a})=N_a^2-K_0\cdot N_a=(1-a^2+a^2+a)-(-1+a)=2.\] So by taking $a\in\ZZ_+$, we obtain an infinite family of classes with index $2$. Note that $$N_{a+1}-N_a=(a+1)(3H-E_1-\cdots-E_9)-E_{10},$$ and one can take $\varepsilon_a>0$ small enough such that $$[\omega_a]:=(1,\frac{1-\varepsilon_a}{3},\cdots,\frac{1-\varepsilon_a}{3},(3a+4)\varepsilon_a)\in\mathcal{P}.$$
    Then $N_{a+1}$ will have minimal area among the classes $\{N_1,\cdots,N_a,N_{a+1}\}$. This means we will need infinitely many minimizers for those non $c_1$-positive $[\omega_a]$'s.

    Also one can not expect the tropical property holds without the reduced condition for $n\geq 9$. This is because the orbit of classes in $U_5$ under the action $D_{K_0}$ becomes infinite when $n\geq 9$.
\end{rmk}

We prove an immediate corollary of the tropical property which will be used in the next section to relate ECH capacities (see also next subsection for a discussion on the algebraic proof). Note that although the tropical property is only for symplectic forms satisfying the reduced condition, we can always view an arbitrary symplectic form as the pullback of a reduced one by a diffeomorphism.  
\begin{corollary}\label{cor:blowup}
Assume $[\omega]$ is a $c_1$-positive symplectic class on $X$. Let $\omega_{\varepsilon}$ be the symplectic form on $X\#\overline{\CC\PP}^2$ with small blowup size $\varepsilon$. Then we have
$$\lim_{\varepsilon\rightarrow 0}f_k(X\#\overline{\CC\PP}^2,\omega_{\varepsilon})=f_k(X,\omega).$$
\end{corollary}

\begin{proof}
    On the one hand, since the natural inclusion $$i:H_2(X;\ZZ)\hookrightarrow H_2(X\#\overline{\CC\PP}^2;\ZZ)$$ preserves the index, the minimizers for $(X,\omega)$ can also be the minimizers for $(X\#\overline{\CC\PP}^2,\omega_{\varepsilon})$ by composing this $i$. So we have $f_k(X\#\overline{\CC\PP}^2,\omega_{\varepsilon})\leq f_k(X,\omega)$. On the other hand, by theorem \ref{thm:main} we only need to consider finitely many minimizers $A_j$'s for $(X\#\overline{\CC\PP}^2,\omega_{\varepsilon})$ to compute its $f_k$. Note that the natural projection $$p:H_2(X\#\overline{\CC\PP}^2;\ZZ)\rightarrow H_2(X;\ZZ)$$ doesn't decrease index. So if $\varepsilon$ goes to zero, by finiteness $\min_j\{\omega_{\varepsilon}(A_j)\}$ will go to $\min_j\{\omega(p(A_j))\}$, which is no less than $f_k(X,\omega)$.
\end{proof}

\subsection{Comparison with the algebraic counterpart}\label{section:comparison}
It's shown in \cite[proposition 3.6]{FM} that any $c_1$-positive symplectic form on rational manifolds can be realized as the K\"{a}hler form of certain complex structure called the ``good generic" complex structure $J$, which means the anti-canonical class $-K_J$ is effective and smooth, and there is no smooth rational curve of self-intersection $-2$. The question surrounding whether symplectic forms without the $c_1$-positive condition could be Kähler remains a deeply mysterious area of study (see \cite{biran} for an exploration in this direction). On the other hand, for other manifolds such as the one point blowup of $S^2$-bundle over $T^2$, one point blowup of $T^4$ or minimal K\"{a}hler surfaces of general type with $b_2^+>1$, there indeed exist non-K\"{a}hler symplectic forms (\cite{nonkahlerruled}, \cite{nonkahlerT4}, \cite{Draghici}) although some other symplectic forms certainly could be K\"{a}hler (unlike Kodaira-Thurston manifold). Note that all complex structures on rational or ruled manifolds have $p_g=0$ and thus any K\"{a}hler triple $(X,\omega,J)$ must be algebraic. This is why we still use `algebraic capacities' in the title and motivates us to wonder to what extent our tamed capacities coincide with Wormleighton's algebraic capacities. 

We first point out that working in the algebraic category, the values for algebraic capacities must be finite. The reason is that we have the ample class $A\in H^{1,1}(X;\ZZ)$ from the polarization. Note that the set
\[\{J\text{-nef}, ind\geq 2k\}\subseteq \text{PD}(H^{1,1}(X;\ZZ))\]
where we do optimization always contains $\text{PD}(nA)$ for large $n$ since it's obviously nef and the index $(nA)^2-K\cdot (nA)$ is a quadratic polynomial of $n$. 
However, when $J$ is not projective, $f_k$ could be $\infty$ due to the triviality of the N\'{e}ron-Severi lattice $H^{1,1}(X;\ZZ)$. Even when $H^{1,1}(X;\ZZ)\neq 0$, there could be the case that the index condition $A^2-K\cdot A\geq 2k\geq 2$ makes the set for optimization empty. For instance, one can take \cite{Zucker}'s example\footnote{It serves as the counterexample to the K\"{a}hler version of Hodge conjecture.} of a non-algebraic complex tori with a negative definite N\'{e}ron-Severi lattice of rank $2$. Such a phenomenon shows that other than rational manifolds which are mainly considered in this paper, there could be a huge distinction between the algebraic capacities using projective complex structures and tamed capacities using possibly non-projective, even non-integrable almost complex structures.



Next, we compare two finiteness results in \cite{Algebraic} to ours. The first one is about viewing algebraic capacities as functions over the big cone. 

\begin{prop}[\cite{Algebraic} proposition 3.1]
    For any projective surface $Y$, if we fix $k$ and view $c_k^{\rm{alg}}(Y,A)$ as a function with variable $A$ inside the big cone $\rm{Big}$$(Y)$, then there is a locally finite chamber decomposition of $\rm{Big}$$(Y)$ such that the function is linear on each chamber.
\end{prop}

 Our theorem \ref{thm:main} stands as a refinement of the previously mentioned proposition when $Y$ is a rational manifold. The key point of our result is that we can improve the local finiteness into global finiteness by restricting to a natural domain parametrizing $c_1$-positive symplectic forms on $Y$. This is rather unexpected to us since in both the study of topology of symplectomorphism groups (\cite{ALLP-stability}) and enumeration of log Calabi-Yau divisors (\cite{Enumerate}) for rational manifolds with varying symplectic forms, the symplectic cone is always decomposed into infinitely many chambers subject to the invariance of certain $\pi_i(\rm{Symp}$$(X,\omega))$ or the amount of log Calabi-Yau divisors. Remarkably, tamed capacities exhibit a distinct pattern due to the finiteness theorem \ref{thm:main}. 
 
 \begin{example}\label{example}
     We can investigate the simple example $Y=\CC\PP^2\#\overline{\CC\PP}^2$ to gain the precise meaning of this refinement. We equip $Y$ with the non-minimal complex structure (first Hirzebruch surface). The real N\'{e}ron-Severi group $\rm{NS}$$(Y)\otimes_{\ZZ} \RR$ can be naturally identified with $H^2(Y;\RR)$. Note that $\rm{Big}$$(Y)$ is the interior of the Mori cone $\overline{\rm{NE}}$$(Y)$ generated by the exceptional class $E$ and the fiber class $H-E$, and the reduced $c_1$-positive symplectic cone $\mathcal{P}^{c_1>0}=\mathcal{P}$ is just the ample cone $\rm{Amp}$$(Y)$, or the interior of the nef cone $\rm{Nef}$$(Y)$ generated by the line class $H$ and the fiber class $H-E$. Since $\rm{Big}$$(Y)$ is open, the algebraic result doesn't guarantee the finiteness of minimizers when the big $\RR$-divisor is moving towards the boundary ray generated by the fiber class $H-E$. See the following figure \ref{fig:cone3}.
 \end{example}

\begin{figure}[H]
		\includegraphics*[width=\linewidth]{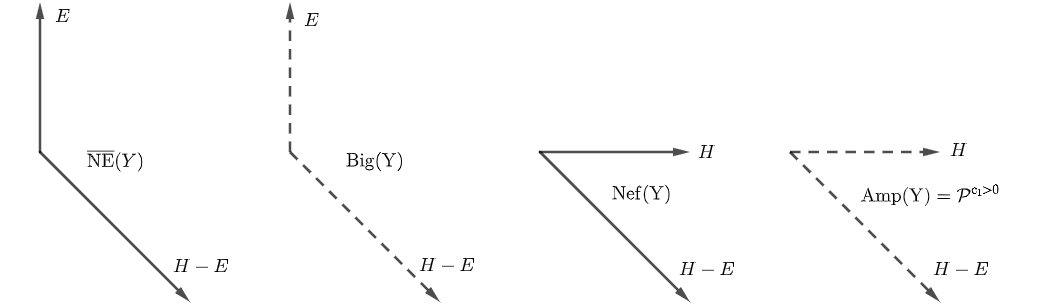}
 
		\caption{Cones in the algebraic settings, where $Y$ is the first Hirzebruch surface. \label{fig:cone3}}
	\end{figure}
 
The second one appears in the context of studying leading and sub-leading asymptotic behavior for capacities.

 \begin{prop}[\cite{Algebraic} lemma 4.15]
    For a smooth pseudo-polarized surface $(Y,A)$ with $A$ being a $\ZZ$-divisor, there exists finitely many nef $\ZZ$-divisors $D_1,\cdots,D_n$ and an integer $K$ such that for all $k\geq K$, the values $c_k^{\rm{alg}}(Y,A)$ must be $A\cdot D_i+jA^2$ for some $i\in\{1,\cdots,n\}$ and $j\in\ZZ_{\geq 0}$.
\end{prop}

    Note that in our finiteness theorem, we fix the $k$-th capacity and allow the symplectic forms to deform, but \cite{Algebraic} fixes a $\ZZ$-divisor (which can be thought as an integral symplectic form) but allows $k$ to be arbitrarily large since the goal there is to study the asymptotic behavior of the capacities.
\begin{theorem}[\cite{Algebraic} theorem 4.2]
    Suppose $Y$ is a smooth algebraic surface, $A$ is a big and nef divisor. Then 
    \[\lim_{k\rightarrow\infty}\frac{c_k^{\rm{alg}}(Y,A)^2}{k}=2A^2.\]
\end{theorem}

We observe that the argument for the preceding theorem relies solely on the numerical attributes of the divisor $A$. For a general symplectic form $\omega$, though it may not correspond to a divisor in the algebraic settings, we can still use the numerical property of the class $[\omega]$ to prove the following variant. The proof below follows the one in \cite{Algebraic} with a mild modification.
\begin{prop}\label{prop:asymp}
    Let $(X,\omega,J)$ be a tamed triple with $b_2^+(X)=1$, then
    \[\lim_{k\rightarrow\infty}\frac{f_k(X,\omega,J)^2}{k}=2[\omega]^2.\]
\end{prop}
\begin{proof}
    For convenience we don't distinguish a class with its Poincar\'{e} dual. Since $b_2^+(X)=1$, we can choose $e_1,\cdots,e_n\in H_2(X;\RR)$ such that $\{[\omega],e_1,\cdots,e_n\}$ forms an orthogonal basis of $H_2(X;\RR)$ and all $e_i^2=-1$.  For any $k\in[-\infty,\infty)$, consider the real version of the $U_1$ set (where we omit $k$ in section \ref{section:hierarchy})
    \[\tilde{U}^k_1:=\{A\in H_2(X;\RR)\,|\,A\cdot C\geq 0\, \forall J\text{-effective }C, A^2-K\cdot A\geq 2k\}\subseteq H_2(X;\RR).\]
    Note that $[\omega]$ and its sufficiently small perturbations must satisfy the $J$-nef condition so that $\tilde{U}^{-\infty}_1$ must contain a conical open neighborhood of the ray $\RR^+[\omega]$. Observe also that for any $k$, a large multiple $N_k[\omega]$ of $[\omega]$ will satisfy the index condition so that it lives in $\tilde{U}^k_1$.
    
    Now suppose $-K=c[\omega]+\sum_{i=1}^nd_ie_i$, then a class $A=a[\omega]+\sum_{i=1}^nc_ie_i\in \tilde{U}^k_1$ will have $\omega(A)=a[\omega]^2$ and satisfy
    \[a^2+ca-\frac{2k+\sum_{i=1}^nb_i(b_i+d_i)}{[\omega]^2}\geq 0.\] Let $\delta:=-\frac{1}{2}\sum_{i=1}^ne_i\in H_2(X;\RR)$, $T_k:=\frac{2k-\frac{1}{4}\sum_{i=1}^nd_i^2}{[\omega]^2}$ and $a_k:=\frac{-c+\sqrt{c^2+4T_k}}{2}$ . Now  we define \[\tilde{f}_k(X,\omega,J):=\inf_{A\in \tilde{U}_1}\omega(A).\] By the property of $\tilde{U}^{-\infty}_1$, for large $k$, $\tilde{A}_k:=a_k[\omega]+\delta\in\tilde{U}^k_1$. Note that $b_2^+(X)=1$ condition also guarantees the existence of a non-zero $J$-effective class. This follows from \cite[proposition 4.3]{LiLiu01}, which says one can always choose a class with positive square and take its large multiple so that it has non-zero SW invariant. Then it's easy to see $\tilde{f}_k(X,\omega,J)=\omega(\tilde{A}_k)=a_k[\omega]^2$ for large $k$ by our choice of the coefficients and the existence of a non-zero $J$-effective class ($a_k$ must be positive). 
    
    To compare $f_k$ and $\tilde{f}_k$, we have to choose an integral approximation $A_k\in H_2(X;\ZZ)$ for each $k$. We require that $A_k$ is `close' to $\tilde{A}_k$ in the sense that all the coefficients of $A_k-\tilde{A}_k$ in terms of the basis $\{[\omega],e_1,\cdots,e_n\}$ have a uniform bound for all $k$. Also, the conical open subset in $\tilde{U}^{-\infty}_1$ makes it possible to let us choose the lattice point $A_k\in \tilde{U}^{-\infty}_1$ for all large $k$ (figure \ref{fig:asym}). Then by a simple computation of the index, it should be clear that $A_k\in \tilde{U}_1^{k+\Delta(k)}$, where the difference term $\Delta(k)$ is an $O(\sqrt{k})$-term by our `closeness'. Therefore \[\lim_{k\rightarrow\infty}\frac{f_k(X,\omega,J)^2}{k}=\lim_{k\rightarrow\infty}\frac{f_{k+O(\sqrt{k})}(X,\omega,J)^2}{k}\leq \lim_{k\rightarrow\infty}\frac{\tilde{f}_k(X,\omega,J)^2}{k} =2[\omega]^2.\]    On the other hand, the obvious relation $\tilde{f}_k(X,\omega,J)\leq f_k(X,\omega,J)$ makes the preceding inequality into equality. 
\end{proof}

\begin{figure}[h]
		\includegraphics*[width=\linewidth]{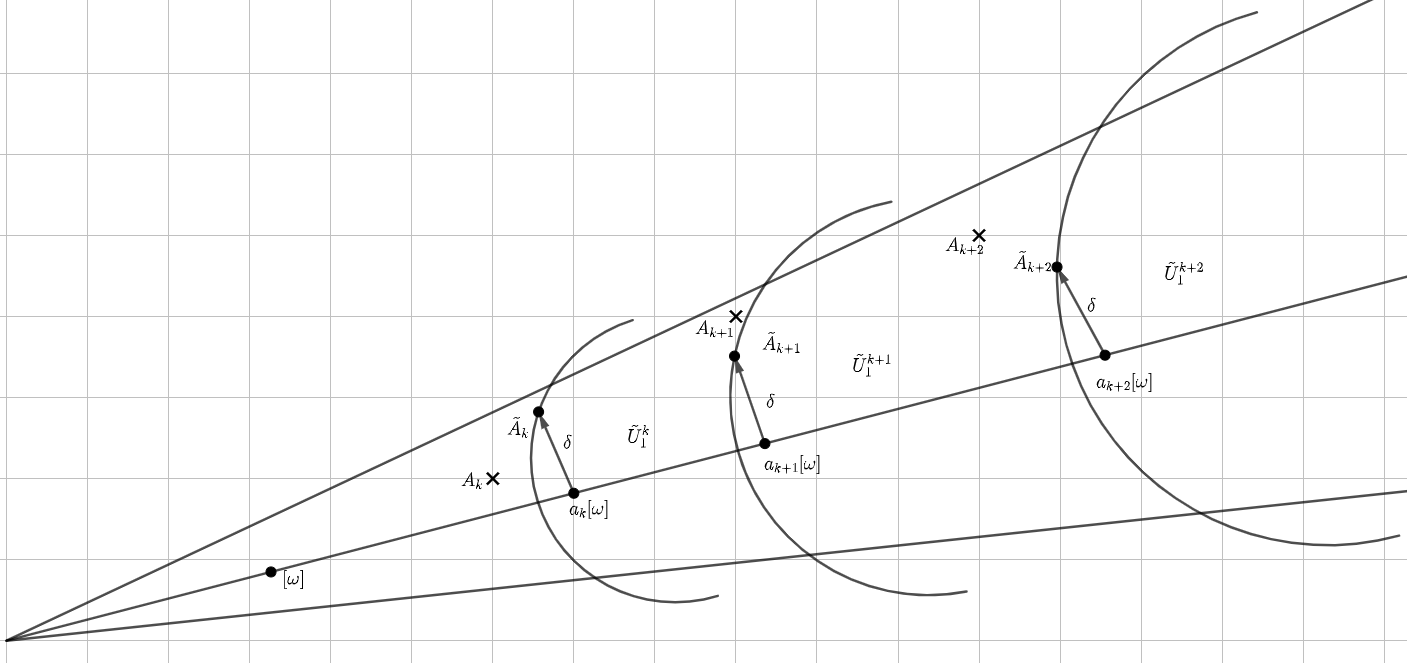}
 
		\caption{This is a conical neighborhood of the ray $\RR^+[\omega]\subseteq \tilde{U}_1^{-\infty}$, where lattice points denote integral homology classes.}\label{fig:asym}
	\end{figure}
   
This asymptotic result gives us a hint that when $k$ is large, the optimizers for $f_k(X,\omega,J)$ are always the integral class very close to a big multiple of $[\omega]$. Therefore we propose the following question regarding $J$-independence.

\begin{question}\label{question:asym}
    Let $(X,\omega)$ be a symplectic $4$-manifold with $b_2^+(X)=1$. Is there a constant $C_{[\omega]}\in\ZZ_+$, which only depends on the cohomology class of $\omega$, such that for all $k\geq C_{[\omega]}$ and any tamed triples $(X,\omega_1,J_1)$, $(X,\omega_2,J_2)$ where $\omega_1,\omega_2$ are cohomologous to $\omega$, $f_k(X,\omega_1,J_1)=f_k(X,\omega_2,J_2)$?
\end{question}

\begin{rmk}
    Without the assumption $b_2^+=1$, this asymptotic behavior will not be true as we already see $f_k$ could be $\infty$. The above argument fails since if $J$ is only tame but not compatible with $\omega$, $[\omega]$ may not live in $H_2(X;\RR)\cap \text{PD}(H_J^+(X))$. Even for compatible pairs, a conical neighborhood of the ray $\RR^+[\omega]\subseteq H_2(X;\RR)\cap \text{PD}(H_J^+(X))$ may not contain any lattice point in $H_2(X;\ZZ)$.
\end{rmk}

   Finally we point out a result that looks quite similar to corollary \ref{cor:blowup}.
 \begin{prop}[\cite{Algebraic} proposition 3.4]
    For smooth projective surface $Y$ with $E\subseteq Y$ a $(-1)$-curve. Let $\pi:Y\rightarrow \overline{Y}$ be the contraction of $E$. Then $$c_k^{\rm{alg}}(Y,\pi^*\overline{A})=c_k^{\rm{alg}}(\overline{Y},\overline{A}).$$
\end{prop}

 In fact, combined with \cite[corollary 3.2]{Algebraic} saying that $c_k^{\rm{alg}}$ is continuous over the big cone, corollary \ref{cor:blowup} could be recovered in the algebraic settings if we know $\overline{A}$ is ample which guarantees $\pi^*\overline{A}$ is big and nef. Since it's not clear whether some non $c_1$-positive symplectic form is K\"{a}hler, it's natural to ask
\begin{question}
    Without assuming $c_1$-positive condition can we still prove corollary \ref{cor:blowup}? 
\end{question}

Taking a closer look at the proof of \cite[proposition 3.1]{Algebraic}, one will find that the finiteness of the set $$\{D\in \text{Nef}(Y):D\cdot A\leq M\}$$
is used, where $A$ is only assumed to be a big divisor. This is an immediate consequence of  Kleiman’s criterion for nefness. Although this has a symplectic analogue like \cite[proposition 4.1.5]{MSJcurvebook}, it's not so straightforward to adjust the argument since a divisor $A$ which is only big might not directly correspond to an actual symplectic form. But by the tropical property one sees how easy it is to obtain a symplectic proof of corollary \ref{cor:blowup}.

\section{Relation with ECH capacities}\label{section:ECH}
From now on, $(X,\omega)$ is further restricted to be a symplectic toric rational manifold. Note that this means $\omega$ must be $c_1$-positive. If the moment polygon is $\Omega$, after an $AGL(2;\ZZ)$ transformation, we can make it sit in the first quadrant. The preimage of $\Omega$ under the map $$\mu:\CC^2\rightarrow \RR^2,\,\,\,\,\,\,\,(z_1,z_2)\mapsto (\pi|z_1|^2,\pi|z_2|^2)$$ is a convex toric domain in the sense of \cite{Danconcave}. We denote it by $X_{\Omega}$, which can also be thought as the complement of $b_2(X)$ components of the toric boundary divisors in $X$. For such a domain there is a sequence of capacities $c_k^{\rm{ECH}}(X_{\Omega})$ associated to it coming from embedded contact homology (\cite{Hutquantitative},\cite{Hutlecture}).  We are going to prove the following result which relates the algebraic capacities of $(X,\omega)$ with the ECH capacities of $X_{\Omega}$:

\begin{prop}\label{prop:ech}
If $(X,\omega)$ has a toric action with moment polygon $\Omega$ then $f_k(X,\omega)=c_k^{\rm{ECH}}(X_{\Omega})$.
\end{prop}

In particular, it implies an interesting phenomenon: if $(X,\omega)$ admits different toric actions with different moment polygons $\Omega_1,\Omega_2$, the ECH capacities of $X_{\Omega_1}$ and $X_{\Omega_2}$ are all the same. This result is not new and has been proved in \cite{BenEhrhart} in a more general sense: $\Omega$ is only assumed to be rational convex which might not be Delzant, and the ECH capacities are identified with the algebraic capacities of a possibly non-smooth projective surface. However, our proof doesn't involve algebraic toric geometry and can be viewed as an application of the previous tropical property.

\subsection{Weight sequence and symplectic cone}\label{section:weightsequence}

Given a Delzant polygon $\Omega$ sitting in the first quadrant, one can associate a sequence of positive numbers $(a,b_1,\cdots,b_k,c_1,\cdots,c_l)$, which is called the weight sequence and denoted by $w(\Omega)$. See \cite{Danconcave} for its definition. We need the following two properties of weight sequence.
\begin{itemize}
    \item When walking clockwise around $\partial\Omega$, one will get the primitive vectors for all the edges. If there is an edge whose vector is $(1,-1)$, then the triangle enclosed by the line extending that edge and $x$-axis, $y$-axis must be the one used to define the first component $a$ in $w(\Omega)$. In this circumstance, the other weights $b_*,c_*$ can be interpreted as the sizes of the corner chopping procedure from the triangle of size $a$ to the Delzant polygon $\Omega$. Therefore, its corresponding closed symplectic manifold must be $\CC\PP^2\#(k+l)\overline{\CC\PP}^2$ with some $c_1$-positive symplectic form $\omega$. After composing some element in $D_{K}$, the weight sequence $w(\Omega)$ will be turned into a unique element in $\mathcal{P}^{c_1>0}$. See figure \ref{fig:weight1} below for an example.
     \begin{figure}[H]
		\includegraphics*[width=\linewidth]{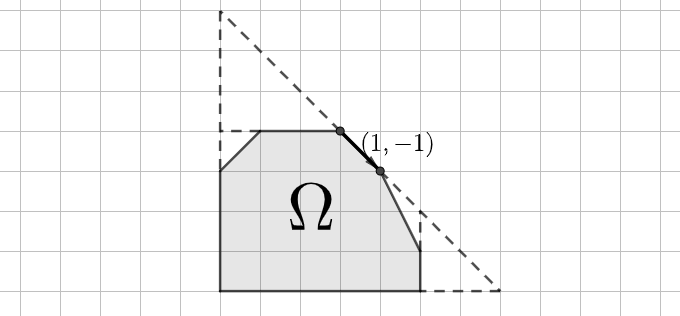}
 
		\caption{The weight sequence of this $\Omega$ is given by $(7,3,1,2,1)$. The closed symplectic toric manifold is $\CC\PP^2\#4\overline{\CC\PP}^2$ with $[\omega]=(7,3,2,1,1)\in\mathcal{P}^{c_1>0}$.\label{fig:weight1}}
	\end{figure}
    \item If there is no such a vector $(1,-1)$, then there must be a unique vertex $p$ of $\Omega$ lying on the segment from $(0,a)$ to $(a,0)$. Now one can perform corner choppings near $p$ of small sizes $\varepsilon_1(t),\cdots,\varepsilon_r(t)$ to get a new Delzant polygon $\Omega'_t$ with one edge's primitive vector being $(1,-1)$. If we let all $\varepsilon_i(t)$ go to $0$ as $t$ goes to $0$, the weight sequence $w(\Omega_t')$ will converge to $w(\Omega)$ after making all of them in a descending order. Therefore, composing some automorphism in $D_{K}$ as above we will obtain an element in $\overline{\mathcal{P}^{c_1>0}}$, the closure of $c_1$-positive cone of the manifold $\CC\PP^2\#(k+l)\overline{\CC\PP}^2$. See figure \ref{fig:weight2}.
    \begin{figure}[H]
		\includegraphics*[width=\linewidth]{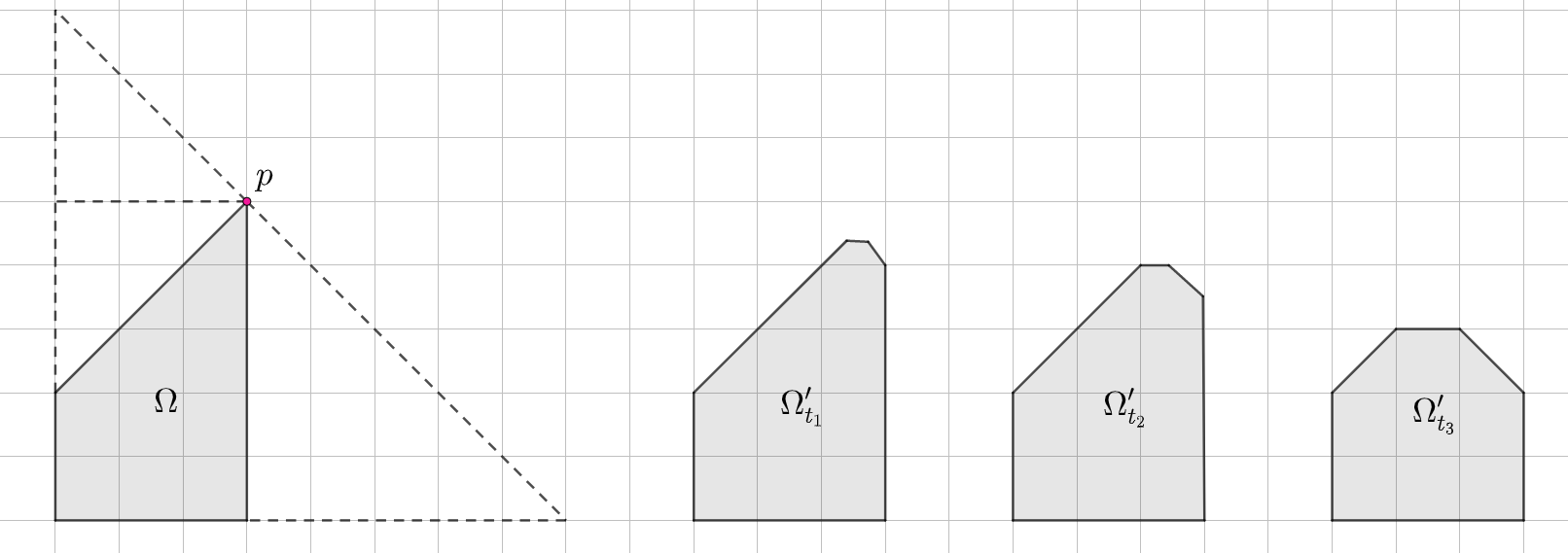}
 
		\caption{The weight sequence of this $\Omega$ is given by $(8,5,3,3)$. One can do toric blowups twice to get those $\Omega'_t$, where the sizes are $\varepsilon_1(t_1)=0.75,\varepsilon_2(t_1)=0.25$; $\varepsilon_1(t_2)=1,\varepsilon_2(t_2)=0.5$; $\varepsilon_1(t_3)=2,\varepsilon_2(t_3)=1$. We see that $w(\Omega'_{t_3})=(5,2,2,1)$, $w(\Omega'_{t_2})=(6.5,3.5,2.5,2)$ and $w(\Omega'_{t_1})=(7,4,2.75,2.25)$ are converging to $w(\Omega)$. By composing an automorphism in $D_K$ on $w(\Omega)$ make it reduced, we will get $(5,2,0,0)$ in the closure of $c_1$-positive cone $\overline{\mathcal{P}^{c_1>0}}$ of the manifold $\CC\PP^2\#3\overline{\CC\PP}^2$.  \label{fig:weight2}}
	\end{figure}
\end{itemize}

\subsection{Proof of $f_k=c_k^{\rm{ECH}}$}
We rely on the formula of \cite{Danconcave}, for ECH capacities in terms of weight sequences:

\begin{lemma}[\cite{Danconcave},\cite{CCFHR},\cite{BenEhrhart}]\label{lem:formula}
    If $\Omega$ has weight sequence $w(\Omega)=(a,b_1,\cdots,b_k,c_1,\cdots,c_l)$, then 
    $$c_k^{\rm{ECH}}(X_{\Omega})=\min_{k_i,m_j\in\ZZ_{\geq 0}}\{c^{\rm{ECH}}_{k+\sum_ik_i+\sum_jm_j}(B(c))-\sum_ic^{\rm{ECH}}_{k_i}(B(a_i))-\sum_jc^{\rm{ECH}}_{m_j}(B(b_j))\}$$
    $$=\min_{x(x+3)-\sum_iy_i(y_i+1)-\sum_jz_j(z_j+1)\geq 2k}\{xc-\sum_i y_ia_i-\sum_jz_jb_j\}$$
\end{lemma}
Note that the first equality already appeared in section 5.2 of \cite{BenEhrhart}. The second equality comes from the fact that 
$$c_k^{\rm{ECH}}(B(a))=da,\,\,\,\text{where }d \text{ satisfies } d(d+1)\leq 2k\leq d(d+3).$$

\begin{proof}[Proof of Proposition \ref{prop:ech}]
    Note that the requirement $$x(x+3)-\sum_iy_i(y_i+1)-\sum_jz_j(z_j+1)\geq 2k$$ of the formula in lemma \ref{lem:formula} is actually the same thing as $ind(A)\geq 2k$ for a homology class $A=xH-\sum_i x_iE_i-\sum_j y_jE_j$. If we are in the first case of section \ref{section:weightsequence}, then the weight sequence corresponds to an element of $\mathcal{P}^{c_1>0}$. So the optimization problems for $f_k$ and $c_k^{\rm{ECH}}$ are completely the same. If we are in the second case, we then apply corollary \ref{cor:blowup} to write $$f_k(X,\omega)=\lim_{t\rightarrow 0}f_k(X\#r\overline{\CC\PP}^2,\omega_t),$$ where $\omega_t$ is given by small blowups of sizes $\varepsilon_1(t),\cdots,\varepsilon_r(t)$ as in section \ref{section:weightsequence}. Note that by the first case, $$f_k(X\#r\overline{\CC\PP}^2,\omega_t)=c_k^{\rm{ECH}}(X_{\Omega'_t}).$$
    Since ECH capacities are continuous with respect to the deformation of the contact boundary, 
    $$\lim_{t\rightarrow 0}c_k^{\rm{ECH}}(X_{\Omega_t'})=c_k^{\rm{ECH}}(X_{\Omega}).$$
    Finally we can combine them to get $$f_k(X,\omega)=c_k^{\rm{ECH}}(X_{\Omega}).$$
\end{proof}

\hfill \break
{\bf Acknowledgement:} The authors would like to thank Julian Chaidez, Dan Cristofaro-Gardiner, Michael Hutchings, Jun Li, Zitian Liu and Jie Min for showing interest in this work and helpful communications.

\bibliographystyle{amsalpha}
	\bibliography{mybib}{}
\end{document}